\documentclass[12pt]{amsart}
\usepackage[colorlinks=true,allcolors=teal]{hyperref}
\usepackage{amsmath,amssymb,latexsym}
\usepackage[mathscr]{eucal}
\usepackage{comment}
\usepackage{enumitem,chngcntr}
\usepackage{tikz-cd}
\usepackage[dvipsnames]{xcolor}
\usepackage{hyperref}

\numberwithin{equation}{section}
\numberwithin{table}{section}
\numberwithin{figure}{section}

\usepackage{soul}
\setstcolor{orange}

\def\cf{cf.~}

\newcommand{\hsp}[1]{{\hbox{\hspace{#1}}}}

\newcounter{letcnt1} 

\newcounter{letcnt2} 
\counterwithin{letcnt2}{letcnt1}

\newcounter{talkcnt} 

\def\b{\beta}  
\def\d{\delta}

\def\z{\zeta}

\def\s{\sigma}
\def\t{\tau}

\def\x{\xi}

\def\tAd{\mathrm{Ad}} \def\tad{\mathrm{ad}}
 
\def\tAut{\mathrm{Aut}}

\def\bC{\mathbb C} 
 
\def\fC{\mathfrak{C}}

\def\td{\mathrm{d}}

 \def\tdim{\mathrm{dim}}
\def\cE{\mathcal E}

\def\tEnd{\mathrm{End}}

\def\bfY{\mathbf{Y}}

  \def\bfG{\mathbf{G}}

\def\tGr{\mathrm{Gr}}
\def\fg{{\mathfrak{g}}}

\def\bi{\mathbf{i}}

  \def\cI{\mathcal I}

\def\tId{\mathrm{Id}} \def\tIm{\mathrm{Im}}
\def\tid{\mathrm{id}}

\def\cJ{\mathcal J}

\def\fk{\mathfrak{k}}

\def\fl{\mathfrak{l}}

\def\cM{\mathcal M}

 \def\cN{\mathcal N}
 
 \def\sfn{\mathsf{n}}

 \def\cO{\mathcal O}
 
\def\sO{\mathscr{O}}

\def\bP{\mathbb P}

\def\fp{\mathfrak{p}}

\def\bQ{\mathbb Q}

\def\bR{\mathbb R}

\def\fS{\mathfrak{S}}  
\def\fT{\mathfrak{T}}  
 
\def\fs{\mathfrak{s}}
\def\bs{\mathbf{s}} 

\def\tSL{\mathrm{SL}} 
\def\Lie{\mathrm{Lie}}

\def\tStab{\mathrm{Stab}}

 \def\tspan{\mathrm{span}}

\def\fu{\mathfrak{u}}
\def\bbV{\mathbb{V}}

\def\cX{\mathcal X}

\def\cY{\mathcal{Y}}

   \def\bZ{\mathbb Z}
 
\def\fz{\mathfrak{z}} 
 
\def\tand{\quad\hbox{and}\quad}
\def\bs{\backslash}

\def\smallb{{\hbox{\small{$\bullet$}}}}

\def\inj{\hookrightarrow}
\def\sur{\twoheadrightarrow}

\def\op{\oplus}
\def\ot{\otimes}

\def\wt{\widetilde}
\def\wh{\widehat}

\newenvironment{a.list}
  {\begin{enumerate}[label=\alph*.,itemsep=3pt,leftmargin=25pt,listparindent=20pt]}
  {\end{enumerate}}

\newenvironment{num.list}
  {
  \begin{enumerate}[itemsep=3pt,leftmargin=25pt,listparindent=20pt,label={\arabic*.}]
  }
  {\end{enumerate}}
\newenvironment{i_list}
  {\begin{enumerate}[label=(\roman*),itemsep=3pt,leftmargin=25pt,listparindent=20pt]}
  {\end{enumerate}}
\newenvironment{i_list_emph}
  {\begin{enumerate}[label=\emph{(\roman*)},itemsep=3pt,leftmargin=25pt,listparindent=20pt]}
  {\end{enumerate}}

\newtheorem{corollary}[equation]{Corollary}
\newtheorem{lemma}[equation]{Lemma}
\newtheorem{proposition}[equation]{Proposition}
\newtheorem{theorem}[equation]{Theorem}

\newtheorem*{theorem*}{Theorem}

\theoremstyle{definition}

\newtheorem*{boldQ*}{Question}
\newtheorem*{boldP*}{Problem}

\theoremstyle{definition}

\theoremstyle{remark}
\newtheorem*{assume*}{Assume}
\newtheorem*{answer*}{Answer}

\newtheorem*{claim*}{Claim}

\newtheorem{definition}[equation]{Definition}
\newtheorem*{definition*}{Definition}
\newtheorem{example}[equation]{Example}
\newtheorem*{example*}{Example}

\newtheorem*{hint*}{Hint}
\newtheorem*{notation*}{Notation}
\newtheorem{remark}[equation]{Remark}
\newtheorem*{remark*}{Remark}
\newtheorem*{remarks*}{Remarks}
\newtheorem*{fact*}{Fact}

\newtheorem*{emphQ*}{Question}
\newtheorem*{emphA*}{Answer}

\def\fB{\mathfrak{B}}

\def\fF{\mathfrak{F}}
\def\bfH{\mathbf{H}}
\def\fH{\mathfrak{H}}
\def\bfL{\mathbf{L}}
\def\fL{\mathfrak{L}}

\def\olM{\overline\cM}

\def\tProj{\mathrm{Proj}}
\def\bfP{\mathbf{P}}

\def\bfQ{\mathbf{Q}}
\def\olS{\overline{S}}
\def\bfU{\mathbf{U}}
\def\fU{\mathfrak{U}}
\def\bfZ{\mathbf{Z}}
\def\fZ{\mathfrak{Z}}

\begin{document}

\title[Completion of two-parameter period maps]{Completion of two-parameter period maps by nilpotent orbits}

\author[Deng]{Haohua Deng}
\email{haohua.deng@dartmouth.edu}

\author[Robles]{Colleen Robles}
\email{robles@math.duke.edu}

\address{Mathematics Department, Dartmouth College, 29 N. Main Street, 6188 Kemeny Hall, Hanover, NH 03755-3551} 

\address{Mathematics Department, Duke University, 120 Science Drive, Box 90320, Durham, NC 27708-0320} 

\thanks{This material is based upon work supported by the National Science Foundation under Grant No. DMS-2304981.}

\date{\today}

\begin{abstract}
We show that every two-parameter period map admits a Kato--Nakayama--Usui completion to a morphism of log manifolds, and the map onto the image is a morphism of compact algebraic spaces.  This result also applies to the case of mixed period maps and we use it to give a construction of generalized N\`eron models.
\end{abstract}



\maketitle
\section{Introduction}

\subsection{Main result} \label{S:mainresult}

Let $S$ be a smooth quasi-projective variety, and let $\bbV \to S$ be a polarized variation of integral Hodge structures ($\bZ$-PVHS) of weight $\sfn$.  Geometrically, (unpolarized) variations of Hodge structure arise as $R^\sfn f_*\bZ/(\mathrm{torsion})$, with $f: \cX \to S$ a smooth proper morphism; by restricting to the local sub-system $\bbV \subset R^\sfn f_*\bZ/(\mathrm{torsion})$ corresponding to the primitive cohomology, we obtain a \emph{polarized} variation of integral Hodge structures.

Fix a point $s_o \in S$, and let $V_\bZ = \bbV_{s_o}$ and $Q : V_\bZ \times V_\bZ \to \bZ$ be the associated lattice and polarization, respectively.  Let $\Gamma \leq \tAut(V_\bZ,Q)$ be the image of the monodromy representation $\pi_1(S,s_o) \to \tAut(V_\bZ,Q)$.  Passing to a finite \'etale cover of $S$, if necessary, we may assume that $\Gamma$ is neat.  Parallel transportation under the Gauss--Manin connection identifies a general fibre $\bbV_s$ with the fixed fibre $\bbV_{s_o} = V_\bZ$.  Under this identification the Hodge structure on $\bbV_s$ defines a Hodge structure $\Phi(s)$ on $V_\bZ$.  The identification $\bbV_s \simeq V_\bZ$ is well-defined up to the action of $\Gamma$.  This yields a period map 
\begin{equation}\label{E:Phi}
  \Phi : S \ \to \ \Gamma\bs D \,,
\end{equation} 
with $D$ a period domain parameterizing $Q$-polarized Hodge structures on $V_\bZ$.  This period map completely describes the $\bZ$-PVHS. 

Without loss of generality $\Phi$ is proper \cite{MR0282990}.  Then the image $\Phi(S)$ of the period map is quasi-projective \cite{MR4557401}.  It is a long-standing problem to construct an algebraic completion of $\Phi$ that encodes Hodge-theoretically meaningful data at infinity (specifically, data encoded by the nilpotent orbits asymptotically approximating $\Phi$) that goes back to conjectures and problems posed by Griffiths in 1970 \cite{MR0258824}.  The problem is well-understood, when $D$ is hermitian (\S\ref{S:completions}).  In general $D$ is not hermitian (for example, the period domain parameterizing the Hodge structures $H^n(X,\bZ)$ of smooth hypersurfaces $X \subset \bP^{n+1}$ of degree $d \ge n+2$ and dimension $n\ge 3$ is not hermitian), and in this case much less is known.  Our main result is an extension theorem for two-parameter period maps.

\begin{theorem} \label{T:main1}
Suppose that $\tdim\,S=2$.  There exists a smooth projective completion $\olS \supset S$, with simple normal crossing divisor $\partial S = \olS \bs S$, and a logarithmic manifold $\Gamma \bs D_\Sigma$ parameterizing $\Gamma$--conjugacy classes of nilpotent orbits on $D$ so that $\Gamma \bs D \inj \Gamma \bs D_\Sigma$ and the period map \eqref{E:Phi} extends to a morphism $\Phi_\Sigma : \olS \to \Gamma \bs D_\Sigma$ of logarithmic manifolds.  The image $\Phi_\Sigma(\olS)$ is a compact algebraic space. 
\end{theorem}

\noindent Theorem \ref{T:main1} is a corollary of Theorem \ref{T:main3}, Kato--Nakayama--Usui's Theorem \ref{T:KNU}, and Usui's Theorem \ref{T:Usui}.  The key technical result of the paper is a certain finiteness statement (Theorem \ref{T:doublecoset}), from which we are able to deduce Theorem \ref{T:main3}.  The proof of the latter (in \S\ref{S:prf-main3}) yields a refinement of Theorem \ref{T:main1}; see Corollary \ref{C:prf} and Remark \ref{R:prf} for details.

Theorem \ref{T:main1} holds in the more general context of variations of graded-polarized mixed Hodge structures (Theorem \ref{T:main1-mixed}).

\begin{remark}[Geometric examples] \label{eg:geometric}
The extensions $\Phi_\Sigma$ are known to exist when $D$ is hermitian, or when $\tdim\,S=1$, cf.~\S\ref{S:completions}.  Examples with $D$ non-hermitian and $\tdim\,S \ge 2$ have been constructed only very recently \cite{Chen1221, MR4441155}.  In those works the $\bZ$-PVHS is induced by a two-parameter family $f:\cX \to S$ of Calabi--Yau three-folds, with Hodge numbers $(1,2,2,1)$, arising as (the mirror of) a complete intersection in a toric variety.  In both cases the authors explicitly exhibit the weak fan necessary to complete the period maps via Kato--Nakayama--Usui's construction \cite{MR2465224}.

Given an arbitrary family $f : \cX \to S$ of Calabi--Yau three-folds, there are four possible (discrete) types of Hodge theoretic degenerations that may arise \cite[Example 5.8]{MR4012553}.  The first author has shown that if there exists a smooth projective completion $\olS \supset S$ so that all degenerations along the simple normal crossing divisor $\olS \bs S$ are of type I or type IV, then (possibly after a proper modification) there is a morphism of locally analytically constructible spaces $\Phi_\Sigma : \olS \to \Gamma \bs D_\Sigma$ that extends the period map \cite{deng2023}.  The constraint on the degeneration type ensures that the nilpotent orbits asymptotically approximating $\Phi$ form a weak fan of restricted type, and the extension of the period map is then a consequence of a slight generalization Kato-Nakayama-Usui's theory (loc.~cit.).
\end{remark}

\begin{remark}
The image $\Phi_\Sigma(\olS)$ of the extension in Theorem \ref{T:main1} is a compact algebraic space.  It is an interesting problem to identify an ample line bundle on $\Phi_\Sigma(\olS)$.  The work \cite{GGR2LB} and \cite[\S5]{GGRinfty} suggests the following natural question: let $\Lambda$ be the augmented Hodge line bundle on $\olS$ (a.k.a.~the Griffiths line bundle).  Do there exist integers $0 < m$ and $0 \le a_i$ so that $\Phi_\Sigma(\olS) = \tProj( \olS , m\Lambda - \sum a_i [S_i])$?
\end{remark}

\subsection{Completions of period maps} \label{S:completions}

Suppose that $D$ is hermitian and that $\Gamma$ is arithmetic.  (Geometric families with $D$ hermitian include curves, principally polarized abelian varieties, and K3 surfaces.)  In this case, there exist compactifications of $\Gamma \bs D$, and completions of period mappings are well-understood.  The two salient constructions are those of Satake--Baily--Borel (SBB) and Ash--Mumford--Rapaport--Tai (AMRT).  Satake constructed a family of topological compactifications of $\Gamma \bs D$ \cite{MR2189882, MR0170356}.  Baily--Borel showed that a minimal Satake compactification $\Gamma \bs D^*$ is projective algebraic (the Hodge line bundle is ample) \cite{MR0216035}.  Moreover, given any smooth projective completion $\olS \supset S$ with simple normal crossing divisor $\olS \bs S$, there is a completion $\Phi^*: \olS \to \Gamma \bs D^*$ \cite{MR0338456}.  The SBB compactification $\Gamma \bs D^*$ is singular (with log canonical singularities \cite{AlexeevBBsing}).  AMRT constructed normalizations $\Gamma \bs D_\Sigma \to \Gamma\bs D^*$ \cite{MR0457437, MR0485875} (introducing toric geometry in the process).  

\smallskip

Much less is known about completions of the period map \eqref{E:Phi} when $D$ is not hermitian.  In the very special case that the period map takes value in a hermitian Mumford--Tate subdomain of $D$, as in \cite{MR1910264, MR2789835, MR1416355, MR3886178, MR2510071}, the SBB and AMRT constructions apply.  And Sommese has shown that every period map with one-dimensional image may be completed \cite{MR0324078}.  In general, the goal is to develop Hodge-theoretically meaningful analogs of SBB and AMRT for arbitrary period maps.  One subtlety here is that one no longer expects a good compactification of the ambient $\Gamma \bs D$.  
Instead the desired analogs are projective algebraic completions of \eqref{E:Phi} that encode the same Hodge-theoretic data as SBB and AMRT.  That data are given by the nilpotent orbits asymptotically approximating $\Phi$.  The two-parameter generalization of SBB is given in \cite{GGLR}, under a local Torelli hypothesis.  Our main result (Theorem \ref{T:main1}) is the two-parameter generalization of AMRT.

Motivation for completing the period map \eqref{E:Phi} includes:  (i) The simple fact that compact spaces are easier to study than noncompact. (ii) Let $\cM$ be the moduli space for smooth varieties $X$ of general type and with fixed Hilbert polynomial.  In a sweeping generalization of the Deligne--Mumford compactification $\overline{\cM}_g$ of the moduli space of curves, Koll\'ar, Shepherd-Barron and Alexeev (KSBA), with contributions of many others, constructed a canonical projective completion $\olM$ \cite{MR3184176}.  However, little is known above the boundary varieties, or the global structure of the moduli space and its boundary $\partial\cM$.  A basic idea is to use the period map and its completions to study $\cM$ and its compactifications \cite{CFPR2022, MR3495110}.

\subsection{Nilpotent orbits}

Fix a period map $\Phi : S \to \Gamma \bs D$ as in \S\ref{S:mainresult}, and a smooth projective completion $\olS \supset S$ with simple normal crossing divisor $\partial S = \olS \bs S$.  To every point $s \in \partial S$, Schmid's nilpotent orbit theorem \cite{MR0382272} associates a nilpotent orbit (Definition \ref{dfn:nilpotentorbit}) that asymptotically approximates (a local lift of) $\Phi$ in a neighborhood of $s$.  So \emph{it is natural to try to complete $\Phi$ using nilpotent orbits.}  

\begin{definition} \label{dfn:nilpotentorbit}
Set $V_\bQ=V_\bZ \ot_\bZ \bQ$, and let $\fg_\bQ$ denote the Lie algebra of linear maps $\x : V_\bQ \to V_\bQ$ such that $Q(\x(u),v)+Q(u,\x(v)) = 0$ for all $u,v\in V_\bQ$.  Let $\check D$ denote the compact dual of $D$.  A \emph{nilpotent orbit pair} $(\s,F)$ consists of:
\begin{i_list}
\item \label{i:cone}
A \emph{rational nilpotent cone} $\s = \tspan_{\bQ_{>0}}\{N_1,\ldots,N_k\}$ of pairwise commuting $N_j \in \fg_\bQ$ (possibly zero).
\item
A filtration $F \in \check D$ so that $N_j(F^p) \subset F^{p-1}$.  In particular, the \emph{nilpotent orbit}
\[
  \nu(z_1,\ldots,z_k) \ = \ \exp(z_1N_1+\cdots z_kN_k) \cdot F
\]
is a horizontal map $\bC^k \to \check D$. 
\item
The nilpotent orbit satisfies $\nu(z_1,\ldots,z_k) \in D$ if $\tIm\,z_j \gg0$.
\end{i_list}
Let $\cN_{\s,F} = \nu(\bC^k) \subset \check D$ denote the image of $\nu$.  In a mild abuse of nomenclature, we will also refer $\cN_{\s,F}$ as a \emph{nilpotent orbit}.
\end{definition}

\begin{example} \label{eg:triv}
If $F \in D$, then $\nu = (0,F)$ is a nilpotent orbit pair.
\end{example}

\begin{example}
If $(\s,F)$ is a nilpotent orbit pair and $F' \in \cN_{\s,F}$, then $(\s,F')$ is also a nilpotent orbit pair and $\cN_{\s,F} = \cN_{\s,F'}$.
\end{example}

\begin{remark} \label{R:W}
Given a nilpotent orbit pair $(\s,F)$, each nilpotent $N\in \s$ determines a rational increasing filtration $W(N)$ of $V_\bQ$ \cite[\S A.3]{MR3290123}.  The filtration satisfies 
\[
  0 \,\subset\, W(N)_{-\sfn} \,\subset\, W(N)_{1-\sfn} \,\subset \cdots \subset\,
  W(N)_{\sfn-1} \,\subset\, W(N)_\sfn \,=\, V_\bQ \,,
\]
and is $\bQ$-isotropic
\begin{equation}\label{E:WisQiso}
   Q\left( W(N)_\ell) \,,\, W(N)_m \right) \ = \ 0 \,,\quad \forall \ \ell+m < 0 \,.
\end{equation}
The filtration $W(N)$ is independent of our choice of $N \in \s$ \cite{MR590823, MR664326}.  So $W(\s)$ is well-defined.  Moreover, $(W(\s)[-\sfn],F)$ is a mixed Hodge structure \cite{MR0382272}; here $\sfn \in \bZ$ is the weight of the Hodge structures on $V_\bZ$.
\end{remark}

The approximating nilpotent orbit $\nu(z_1,\ldots,z_k)$ in Schmid's theorem depends on a choice of local coordinates about $s \in \partial S$.  One may think of the theorem as describing how to complete the period map locally about $s$, at least set-theoretically.  The challenge is to give a global, algebraic completion of the period map that retains all the information in the nilpotent orbit that is independent of the choice of local coordinates.  That information is the $\Gamma$--equivalence class of the pair $(\s,\cN_{\s,F})$.  In the case that $\Gamma \bs D$ is locally hermitian symmetric, these equivalence classes are precisely the objects parameterized by the AMRT compactification $\Gamma \bs D_\Sigma$, \cite{MR605337}.  Kato--Usui \cite{MR2465224} and Kato--Nakayama--Usui \cite{MR3084721} developed a theory generalizing the construction of AMRT to arbitrary period domains.  In order to obtain the desired global extension $\Phi_\Sigma : \olS \to \Gamma \bs D_\Sigma$ of the period map \eqref{E:Phi} it is necessary to exhibit a ``weak fan'' $\Sigma$ that is compatible with $\Phi$, cf.~\S\ref{S:KNU}.  When the fan exists, the Kato--Usui space $\Gamma \bs D_\Sigma$ parameterizes $\Gamma$--conjugacy classes of pairs $(\s,\cN_{\s,F})$.  Such fans are known to exist when $D$ is hermitian (one recovers AMRT) and when $\tdim\,\Phi(S)=1$.  In general, it is quite difficult to demonstrate the existence of a suitable weak fan; to the best of our knowledge, the list in Example \ref{eg:geometric} is exhaustive.  The main technical result of this paper is the existence of a compatible weak fan when $\tdim\,S=2$ (Theorem \ref{T:main3}). 

\subsection{Relationship to the conjectural analog of SBB}
Before turning to the details of Kato-Nakayama-Usui (KNU) construction, we wish to comment on the relationship to an analog of SBB.

Each nilpotent orbit pair $(\s,F)$ determines a limiting mixed Hodge structure.  Quotienting out the extension data, we obtain a polarized Hodge structure $\phi(\s,F)$.  This Hodge structure is independent of the choice of $F \in \cN_{\s,F}$, so that $\phi(\s,\cN_{\s,F})$ is well-defined.  The map $(\s,\cN_{\s,F}) \mapsto \phi(\s,\cN_{\s,F})$ is $\Gamma$--equivariant.  (See \cite{GGRinfty} for details.)  The AMRT normalization $\Gamma \bs D_\Sigma \to \Gamma \bs D^*$ of SBB is the map induced by $\phi$.  This is the sense in which AMRT and SBB are the maximal and minimal Hodge--theoretically meaningful compactifications of $\Gamma\bs D$, respectively. 

A conjectural analog of SBB has been defined for arbitrary period maps \cite{GGRinfty}.  The analog is an extension $\Phi^* : \olS \to \wp^*$ of the period map \eqref{E:Phi}.  Here $\wp^* = \Phi^*(\olS)$ is a compact Hausdorff space compactifying the image $\wp = \Phi(S)$ of the period map, and the map $\Phi^*$ is a proper, continuous extension of $\Phi$.  If the period map also admits a KNU extension $\Phi_\Sigma : \olS \to \Gamma\bs D_\Sigma$ (for example, as in Theorem \ref{T:main1}) then the map $\Phi^*$ will factor through $\Phi_\Sigma$, and in this case $\wp^*$ is the image of $\Phi_\Sigma(\olS)$ under the map induce by $\phi$. When $D$ is hermitian, $\wp^*$ is precisely the closure of $\wp$ in the SBB compactification $\Gamma\bs D^*$, and $\Phi^*$ is Borel's extension.  In general, it is conjectured that $\wp^*$ is a complex variety (in which case the Riemann extension theorem implies $\Phi^*$ is holomorphic); that $\Lambda$ is semi-ample on $\olS$; and that $\Phi^*$ factors through $\olS \to \tProj(\olS,\Lambda)$, with $\tProj(\olS,\Lambda) \to \wp^*$ finite.  The conjecture is known to hold in three situations: (i) when $D$ is hermitian (in which case we recover SBB, and have the stronger statement $\tProj(\olS,\Lambda) = \wp^*$); (ii) when $\tdim\,\wp = 1$ \cite{MR0324078}; and (iii) when $\tdim\,S=2$ and $\Phi$ satisfies local Torelli \cite{GGLR}.

\subsection{The Kato-Nakayama-Usui construction} \label{S:KNU}

The nilpotent orbit pair $(\s_s,F_s)$ associated to $s \in \partial S$ by the nilpotent orbit theorem is well-defined up to the action of $\Gamma$.  Let 
\[
  \Sigma_{\Phi,\olS} 
  \ = \ \{ \s_s \ | \ s \in \partial S \} \,\cup\, \{0\} 
\]
denote the collection of all nilpotent cones arising in this way, plus the trivial cone.  

\begin{remark} 
(i) If $\s = \tspan_{\bQ_{>0}}\{ N_1 , \ldots , N_k \} \in \Sigma_{\Phi,\olS}$, then $\Sigma_{\Phi,\olS}$ also contains every face $\s_I = \tspan_{\bQ_{>0}}\{ N_i \ | \ i \in I \}$ of $\s$; here $I$ runs over all subsets of $\{1,\ldots,k\}$.  By convention $\s_\emptyset$ is the zero cone $\{0\}$.

(ii) The nilpotent cones $\s_s$ are locally constant on the smooth strata of $\partial S$.  Consequently $\Gamma$ acts on $\Sigma_{\Phi,\olS}$ with finitely many orbits.
\end{remark}

\begin{definition}[{\cite{MR2465224}}]\label{dfn:SigmaKU}
Let $\Sigma$ be a collection of rational nilpotent cones $\s \in \fg_\bQ$, that is closed under taking faces.  We say that the local monodromies of $(\Phi,\olS)$ are \emph{in the directions in $\Sigma$} if for every pair $(\s_s,F_s)$ there exists a unique minimal $\tau_s \in \Sigma$ so that $\s_s \subset \tau_s$ and $(\tau_s,F_s)$ is a nilpotent orbit pair.
\end{definition}

\begin{definition} \label{dfn:weakfan}
A collection $\Sigma$ of rational nilpotent cones $\s \in \fg_\bQ$, that is closed under taking faces, is a \emph{weak fan} if for all $\s,\tau \in \Sigma$ we have $\s = \tau$ whenever:
\begin{i_list}
\item
the intersection $\s \cap \tau$ is nonempty, and 
\item
there exists $F \in \check D$ so that both $(\s,F)$ and $(\tau,F)$ are nilpotent orbit pairs.
\end{i_list}
\end{definition}

\noindent Given a collection $\Sigma$ of rational nilpotent cones define
\[
  D_{\Sigma} \ = \ 
  \{ \cN_{\s,F} \ | \ \s \in \Sigma
  \hbox{ and } (\s,F) \hbox{ is a nilpotent orbit pair}\} \,.
\]
Note that $D \subset D_{\Sigma}$ (Example \ref{eg:triv}).  The quotient $\Gamma \bs D_{\Sigma} \supset \Gamma \bs D$ parameterizes $\Gamma$--conjugacy classes of nilpotent orbits $\cN_{\s,F}$ with $\s \in \Sigma$.

\begin{theorem}[Kato--Nakayama--Usui {\cite{MR3084721}}] \label{T:KNU}
If $\Sigma$ is a weak fan, then $\Gamma \bs D_{\Sigma}$ admits the structure of a logarithmic manifold.  If the local monodromies of $(\Phi,\olS)$ are in the directions in $\Sigma$ \emph{(Definition \ref{dfn:SigmaKU})}, then the period map $\Phi$ extends to a morphism $\Phi_\Sigma : \olS \to \Gamma \bs D_{\Sigma}$ of logarithmic manifolds.
\end{theorem}

\begin{remark}\label{R:KNU}
The extension $\Phi_\Sigma$ maps $s\in \olS$ to the $\Gamma$ conjugacy class of the pair $(\tau_s, F_s)$ in Definition \ref{dfn:SigmaKU}.
\end{remark}


\begin{theorem}[Usui {\cite{MR2237263}}] \label{T:Usui}
The image $\Phi_\Sigma(\olS)$ is a compact algebraic space.
\end{theorem}

\noindent Theorem \ref{T:main1} is a corollary of Theorem \ref{T:KNU}, Theorem \ref{T:Usui} and Theorem \ref{T:main3}.

\begin{theorem} \label{T:main3}
Assume $\tdim\,S=2$.  There exists a smooth projective completion $\olS \supset S$, with simple normal crossing divisor $\partial S = \olS \bs S$, and a weak fan $\Sigma$ so that the local monodromies of $(\Phi,\olS)$ are in the directions in $\Sigma$. 
\end{theorem}

\begin{remark} \label{R:main3}
Theorem \ref{T:main3} is proved in \S\ref{S:prf-main3}.  The basic idea is to fix some smooth projective completion $\tilde S \supset S$ with simple normal crossing boundary divisor, and show that after $\mathrm{Ad}_{\Gamma}$-finite subdivision of the cones in $\Sigma_{\Phi,\tilde S}$ we obtain a weak fan $\Sigma''_{\Phi,\tilde S}$ (Proposition \ref{P:wf}).  We then show (in \S\ref{S:prf-main3}) that there is a smooth projective completion $\olS \supset S$, that is obtained from $\tilde S$ by a finite sequence of blow-ups, and with the property that the local monodromies of $(\Phi,\olS)$ are in the directions in $\Sigma = \Sigma''_{\Phi,\tilde S}$.
\end{remark}

\begin{remark}
Theorem \ref{T:doublecoset} below is the key result of the paper, and there is no constraint on the dimension of the cones in that theorem or its corollary.  However, the restriction to dimension at most two is necessary for our proof of Proposition \ref{P:wf}.  The difficulty generalizing the lemma is due to the combinatorial complications that arise when $\tdim\,\s \ge 3$.  In some cases these complications can be circumvented; see \cite{deng2023}.
\end{remark}

\subsection*{Organization of the paper.}
In Section \ref{S:finiteness} we establish the finiteness theorem (Theorem \ref{T:doublecoset}) which is the main input for the construction of weak fans in Section \ref{S:KNU}; In Section \ref{S:weakfan} we use Theorem \ref{T:doublecoset} to produce a weak fan compatible with any two-parameter period map. In Section \ref{S:VMHS} we clarify our construction also applies to mixed period maps and the construction of generalized Néron models. Section \ref{S:prf-BKT+} is devoted to the proof of Theorem \ref{T:BKTplus} and is the main technical part of the paper. Section \ref{S:appendix} is used as an appendix for necessary background in Siegel sets and fundamental sets.

\subsection*{Acknowledgments}

The authors are indebted to Matt Kerr for several illuminating and enriching discussions.  We also thank Ben Bakker, Bruno Klingler, Phillip Griffiths, Chikara Nakayama and Christian Schnell for related conversations.
\section{A finiteness property} \label{S:finiteness}

Fix any two rational nilpotent cones $\s,\tau \subset \fg_\bQ$, as in Definition \ref{dfn:nilpotentorbit}\ref{i:cone}.  Given $\gamma \in \Gamma$, set $\tau_\gamma = \tAd_\gamma(\tau)$.  Define
\begin{subequations}\label{SE:Ist}
\begin{equation}
  \Gamma_{\s,\tau} \ = \ 
  \left\{ \gamma \in \Gamma \ \left| \ 
  \begin{array}{l}
  \s \cap \tau_\gamma \not= \emptyset \,,\ \exists \ F_\gamma \in \check D
  \hbox{\ s.t.\ } (\s,F_\gamma) \hbox{ and}\\
  (\tau_\gamma , F_\gamma) 
  \hbox{ are nilpotent orbit pairs} 
  \end{array}
  \right. \right\} \,,
\end{equation}
and let
\begin{equation}
\cI_{\s,\tau} \ = \ 
  \{ \sigma \cap \tau_\gamma \ | \ 
  \gamma \in \Gamma_{\s,\tau} \} 
\end{equation}
\end{subequations}
be the associated collection of rational nilpotent cones.

Let  
\[
  \bfG \ = \ \tAut(V_\bQ,Q) \ = \ 
  \{ g \in \tAut(V_\bQ) \ | \ Q(gu,gv) = Q(u,v) \,,\ 
  \forall \  u,v \in V_\bC \}
\]
be the $\bQ$--algebraic group of invertible linear maps $V_\bQ \to V_\bQ$ preserving $Q$.  The Lie group 
\[
  G \ = \ \bfG(\bR)\footnote{In the rest of the paper, we will use the boldface character to define an algebraic group over $\bQ$, and the corresponding normal character to define the $\bR$-points of this group.}
\]
of real points is the automorphism group of $D$.  Let $\bfZ_\s$ and $\bfZ_\tau$ denote the $\bQ$-algebraic subgroups of $\bfG$ centralizing $\s$ and $\tau$, respectively.  We have $\Gamma \leq \bfG(\bZ)$.  Let 
\[
  \pi : \bfG(\bZ)  \ \to \ 
  \bfZ_\s(\bZ) \bs \bfG(\bZ) / \bfZ_\tau(\bZ)
\]
denote the natural projection to the double coset space.

\begin{theorem}\label{T:doublecoset}
The image $\pi(\Gamma_{\s,\tau})$ is finite.
\end{theorem}

\noindent If $\pi(\gamma) = \pi(\gamma')$, then $\s \cap \tau_{\gamma} = \s \cap \tau_{\gamma'}$.  Consequently, we have

\begin{lemma} \label{L:Ifinite}
The map $\Gamma_{\s,\tau} \to \cI_{\s,\tau}$ given by $\gamma \mapsto \s \cap \tau_\gamma$ factors through $\pi$ to define a surjection
\[
  \pi(\Gamma_{\s,\tau}) \ \sur \ 
  \cI_{\s,\tau} \,.
\]
\end{lemma}

\begin{corollary} \label{C:Ifinite}
The collection $\cI_{\s,\tau}$ is finite.
\end{corollary}

\noindent The remainder of \S\ref{S:finiteness} is occupied with the proof of Theorem \ref{T:doublecoset}.  There is no bound on the dimension of the cones $\s$ in this argument.

\begin{remark}
The alert reader will notice that these three results (Theorem \ref{T:doublecoset}, Lemma \ref{L:Ifinite} and Corollary \ref{C:Ifinite}) all hold under the assumption that $\Gamma$ is contained in \emph{some} arithmetic subgroup of $\mathbf{G}(\bQ)$.  In particular, we need not assume $\Gamma$ is the image of the monodromy representation.  (For example, $\Gamma$ may be a thin subgroup.)  Alternatively the reader might take $\Gamma = \mathbf{G}(\bZ)$, and deduce the general statements as straightforward corollaries of this special case.
\end{remark}

\subsection{Fundamental sets} \label{S:fundsets}

We introduce the concept of fundamental sets in this section. The relationship between fundamental sets and Siegel sets will be reviewed in detail in Section \ref{S:appendix}. Set 
\[
  B_\s \ = \ \{ F \in \check D \ | \ (\s,F) 
  \hbox{ is a nilpotent orbit pair with $\bR$-split } 
  (W(\s)[-\sfn],F) \} \,.
\]
The centralizer $Z_\s = \bfZ_\s(\bR)$ acts on $B_\s$ with finitely many orbits, each of which is open and closed \cite{MR3474815}.

The proof of Theorem \ref{T:doublecoset} makes use of the notation of a fundamental set for the action of $\bfZ_\s(\bZ)$ on $B_\s$.

\begin{definition}[{\cite[\S1.4]{MR0204533}}] \label{dfn:fundset}
Let $\bfH$ be a $\bQ$--algebraic group.  We say $\fH \subset H = \bfH(\bR)$ is a \emph{fundamental set} for $\bfH(\bZ)$ if:
\begin{i_list}
\item\label{i:maxcpt}
$\fH K = \fH$ for some maximal compact subgroup $K \subset H = \bfH(\bR)$,
\item $\bfH(\bZ)\, \fH = H$, and 
\item
$\{ \gamma \in \bfH(\bZ) \ | \ \fH \,\cap\,\gamma\fH \not= \emptyset \}$ is finite.  
\end{i_list}
\end{definition}

\begin{remark}[{\cite[\S1.4]{MR0204533}}]  \label{R:fundset}
If $K_0 \leq K$ is a subgroup, then $\fF = \fH/K_0$ is a \emph{fundamental set} for the action of $\bfH(\bZ)$ on $X = H/K_0$ in the sense that 
\item
\begin{i_list}
\item 
$\bfH(\bZ)\, \fF = X$, and 
\item
$\{ \gamma \in \bfH(\bZ) \ | \ \fF \,\cap\,\gamma\fF \not= \emptyset \}$ is finite. 
\end{i_list}
\end{remark}

Let $\bfH$ be a $\bQ$-algebraic group.  Fundamental sets $\fH \subset H$ for $\bfH(\bZ)$ exist.  We briefly review here those properties that will be needed in the proof of Theorem \ref{T:doublecoset}.   Let $\bfU \leq \bfH$ denote the unipotent radical, the maximal connected unipotent normal subgroups of $\bfH$ \cite[\S11.21]{MR1102012}. There is a $\bQ$-reductive Levi factor $\bfL \leq \bfH$ so that $\bfH = \bfU \rtimes \bfL$ \cite{MR0092928}.  

\begin{lemma}\label{L:GFundSet}
The group $\bfU(\bZ) \rtimes \bfL(\bZ)$ has finite index in $\bfH(\bZ)$. As a consequence, fundamental sets $\fH$ of $H = \bfH(\bR)$ for $\bfH(\bZ)$ are realized as $\fH = \fU \times \fL$, with $\fU \subset U = \bfU(\bR)$ a fundamental set for $\bfU(\bZ)$, and $\fL \subset L = \bfL(\bR)$ a fundamental set for $\bfL(\bZ)$
\end{lemma}
\begin{proof}
    The first statement is \cite[Corollary 6.4]{MR0147566}. For the second, we make the following observations.
\begin{itemize} 
\item
The quotient $\bfU(\bZ)\bs U$ is compact \cite[\S1.4]{MR0204533}, so that $U$ admits a bounded fundamental set $\fU \subset U$ for the action of $\bfU(\bZ)$.  (Because $U$ is unipotent, the maximal compact subgroup is trivial.)
\item
There exists a finite subset $C \subset \bfL(\bQ)$ and a Siegel set $\fS \subset L$, so that $\fL = C \cdot\fS$ is a fundamental set for the action of $\bfL(\bZ)$ on $L$ \cite[Theorem 1.10]{MR0204533}.
\item 
A maximal compact subgroup of $L$ is also a maximal compact subgroup of $H$.
\item
If $(u_i,\lambda_i) \in U \rtimes L$, then the multiplication in $\bfH$ is given by 
\[
  (u_1,\lambda_1) \cdot (u_2,\lambda_2)
  \ = \ ( u_1 (\lambda_1 u_2\lambda_1^{-1}) \,,\, 
  \lambda_1\lambda_2) \,.
\]
\end{itemize}
From this discussion, it is straight-forward to deduce that $\fH = \fU \times \fL$ is a fundamental set for $\bfH(\bZ)$.  
\end{proof}

\subsection{Normalization I: weight filtration} \label{S:nI}

Given $\gamma \in \Gamma_{\s,\tau}$, fix $0\not= M_\gamma \in \s \cap \tAd_\gamma(\tau)$.  Independence of the weight filtration \cite{MR590823, MR664326} implies
\[
  W(\s) \ = \ W(M_\gamma) \ = \ W(\tAd_\gamma(\tau)) 
  \ = \ \gamma W(\tau) \,.
\]
Set $W = W(\s)$.

Without loss of generality, we may assume that $\tid \in \Gamma_{\s,\tau}$.  Else, fix $\gamma_o \in \Gamma_{\s,\tau}$; and replace $\tau$ with $\tAd_{\gamma_o}(\tau)$, and $\gamma$ with $\gamma\,\gamma_o^{-1}$.  The intersections $\s \cap \tAd_\gamma(\tau)$ are invariant under this operation, and we have 
\begin{equation}\label{E:norm}
  \Gamma_{\s,\tau} \ \subset \ \bfP_W(\bZ)
\end{equation}
where $\bfP_W$ is the $\bQ$-algebraic parabolic subgroup preserving the weight filtration.

Recall that there exists an element $\delta\in \mathfrak{g}_{\bR}$ commuting with \emph{every} morphism of $(W[-\sfn],F_\gamma)$, including the elements of both $\sigma$ and $\mathrm{Ad}_{\gamma}\tau$, so that $(W[-\sfn], e^{-\bi\delta} F_\gamma)$ is $\bR$-split \cite[(2.20)]{MR840721}. Consequently, by replacing $F_\gamma$ with $e^{-\bi\delta}F_\gamma$, we may assume $(W[-\sfn],F_\gamma)$ is $\bR$--split.

\subsection{Normalization II: rationality}

Given $F \in B_\s$, let 
\begin{subequations}\label{SE:ds-V}
\begin{equation}
  V_\bC \ = \ \op\, V^{p,q}_{W,F} \,, \quad
  \overline{V^{p,q}_{W,F}} \ = \ V^{q,p}_{W,F}
\end{equation}
be the Deligne splitting of the mixed Hodge structure $(W[-\sfn],F)$ \cite[\S2]{MR840721}.  We have 
\begin{equation} 
  F^k \ = \ \bigoplus_{p\ge k} V^{p,q} _{W,F}
  \tand
  W[-\sfn]_\ell \ = \ \bigoplus_{p+q \le \ell} V^{p,q}_{W,F} \,.
\end{equation}
\end{subequations}
The induced mixed Hodge structure on $\fg$, cf.~\cite[\S7.5.2]{MR3290129}, has Deligne splitting 
\begin{subequations}\label{SE:ds-g}
\begin{equation}
  \fg_\bC \ = \ \op\,\fg^{p,q}_{W,F}\,,\quad
  \fg^{p,q}_{W,F} \ = \ 
  \{ \x \in \fg_\bC \ | \ 
  \x(V^{r,s}_{W,F}) \subset V^{p+r,q+s}_{W,F} \} \,,
\end{equation}
is also $\bR$--split $\overline{\fg^{p,q}_{W,F}} = \fg^{q,p}_{W,F}$, and $\s,\tau \subset \fg^{-1,-1}_{W,F}$.  The complex Lie subalgebras of $\fg_\bC$ stabilizing $W$ and $F$ are 
\begin{equation}
  \fp_{W,\bC} \ = \ \bigoplus_{p+q \le 0} \fg^{p,q}_{W,F}
  \tand
  \fp_{F,\bC} \ = \ \bigoplus_{p \ge 0} \fg^{p,q}_{W,F} \,,
\end{equation}
\end{subequations}
respectively.  The endomorphism $Y : V_\bC \to V_\bC$ defined by 
\begin{equation}\label{E:Y-LMHS}
Y(v) = (p+q-\sfn)v, \ \forall v \in V^{p,q}_{W,F}
\end{equation}
is an element of the Lie algebra
\begin{subequations}\label{SE:S}
\begin{equation}
  \fs \ = \ \fg_\bR \,\cap\, \fg^{0,0}_{W,F}
\end{equation}
of the stabilizer
\begin{equation}
  \tStab_G(W,F) \ = \ 
  \{ g \in G \ | \ g(W_\ell) = W_\ell \,,\ g(F^k) = F^k \}
  \ \leq \ P_W
\end{equation}
\end{subequations}
of the weight and Hodge filtrations \cite[\S2]{MR840721}.

Note that $\bfP_W$ naturally acts on $\tGr^W_\ell = W_\ell/W_{\ell-1}$.  The unipotent radical $\bfU_W \leq \bfP_W$ is the subgroup acting trivially on the $\tGr^W_\ell$.  Since $W$ is determined by $\s$, we have
\begin{equation}\label{E:ZinP}
  \bfZ_\s \,,\ \bfZ_\tau \ \leq \ \bfZ_{M_\tid} 
  \ \leq \ \bfP_W \,.
\end{equation} 
The unipotent radicals of $\bfZ_\s$ and $\bfZ_\tau$ are $\bfU_\s = \bfU_W \cap \bfZ_\s$ and $\bfU_\tau = \bfU_W \cap \bfZ_\tau$, respectively \cite[Theorem 5.2]{MR3474815}.

\begin{lemma} \label{L:rat}
There exists $g \in U_\s \cap U_\tau$ so that $\tAd_g(Y) \in \fg_\bQ$.
\end{lemma}

\noindent From this point on, we specialize to the case that $F = F_\tid$; that is, $Y$ is the endomorphism $V_\bC \to V_\bC$ acting on $V^{p,q}_{W,F_\tid}$ by the eigenvalue $(p+q-\sfn)$.

\begin{corollary}\label{C:rat}
Replacing $F_\tid$ with $g F_\tid$, we may assume that $Y$ is rational.
\end{corollary}

\begin{proof}[Proof of Lemma \ref{L:rat}]
Fix $M \in \sigma \cup \tau$.  Let $U_M$ be the unipotent radical of the centralizer $Z_M$ of $M$ in $G$.  Then $U_M = \exp( \fu_M)$ with $\fu_M = \mathrm{ker}(\tad_M) \cap \mathrm{im}(\tad_M))$, \cite[\S1.1 or Theorem 2.9]{MR3474815}.  The group $U_M$ acts simply transitively on the set
\begin{equation}\label{E:kostantY}
  \mathcal{Y}_M = 
  \{ Y' \in \mathrm{im}(\tad_M) \ | \ \tad_M(Y') = 2M \} 
  \ = \ Y' \,+\, \fu_M\,,
\end{equation}
\cite[Corollary 3.5 and Theorem 3.6]{MR0114875}.  The centralizer of $\s\cup\tau$ is 
\[
  Z_{\s \cup \tau} \ = \ 
  \bigcap_{M \in \s \cup \tau} Z_M
  \ = \ 
  Z_\s \,\cap \, Z_\tau  \,,
\]
and the unipotent radical of this group is 
\[
  U_{\sigma \cup \tau} \ = \
  U_\sigma \,\cap\, U_\tau \ = \ 
 \bigcap_{M \in \s \cup \tau} U_M \ = \ 
  \exp(\fu_{\sigma \cup \tau}) \,,
\]
with 
\[
  \fu_{\sigma \cup \tau} \ = \ 
  \bigcap_{M \in \sigma \cup \tau} \fu_M \,.
\]
Kostant's result \eqref{E:kostantY} implies that $U_{\s\cup\tau}$ acts simply transitively on  
\[
  \cY \ = \ \bigcap_{M \in \s \cup \tau}
  \cY_M \ = \ 
  \bigcap_{M \in \s\cup \tau} 
  \{Y' \in \mathrm{im}(\tad_M) \ | \ \tad_M(Y') = 2M \} 
  \ = \ Y' \,+\, \fu_{\sigma \cup \tau}
  \,.
\]
Note that $\cY$ is the set of solutions $x$ to an inhomogeneous system of linear equations $Ax=b$ defined over $\bQ$.  Since there exists a real solution $Y$, there necessarily exists a rational solution.
\end{proof}

\subsection{Normalization III: fundamental sets $\fB_\s \subset B_\s$} 

Let $F_\tid \in B_\s$ and $Y \in \fg_\bQ$ be as in Corollary \ref{C:rat}.  The centralizer 
\[
  \bfL_W \ = \ \{ g \in \bfP_W \ | \ \tAd_g(Y) = Y \}
\]
is a Levi factor, by Corollary \ref{C:rat} and \cite[\S2.2]{MR3505643}.  Moreover, $\s \subset \fg^{-1,-1}$ implies that $[Y,N] = -2N$ for all $N \in \s$, and this implies that 
\[
  \bfL_\s \ = \ \bfZ_\s \,\cap\, \bfL_W
\]
is a Levi factor of $\bfZ_\s$; we have 
\begin{equation}\label{E:UL=Z}
  \bfZ_\s \ = \ \bfU_\s \,\rtimes\, \bfL_\s 
\end{equation}
as $\bQ$--algebraic groups.

\begin{lemma}[{\cite[p.~314]{MR0382272}}]\label{L:uniqueK}
Given $\varphi \in D$ in $G$, there is a unique maximal compact subgroup $K \leq G$ that contains the stabilizer $\tStab_G(\varphi)$ of $\varphi$ in $G$.  The Cartan involution $\theta = \theta_K : G \to G$ corresponding to $K$ is given by conjugation by the Weil operator $\varphi(\bi) \in G$ of $\varphi$.  In particular, if $\fg_\bC = \bigoplus \fg^{p,-p}_\varphi$ is the Hodge decomposition induced by $\varphi$, then $\td\theta$ acts on $\fg^{p,-p}$ by the eigen-value $\bi^{2p} = (-1)^p$.
\end{lemma}

\begin{lemma} \label{L:compatibleK}
Fix $M \in \sigma$ and $F \in B_\s$.  Set $\varphi = e^{\bi M} F \in D$.  Let $K \leq G$ be the unique maximal compact subgroup containing the stabilizer of $\varphi$, and let $\theta : G \to G$ be the associated Cartan involution.  
Define $K_W = P_W \cap K$ and $K_\s = Z_\s \cap K$.  Then:
\begin{i_list_emph}
\item \label{i:KW} 
$L_W = P_W \cap \theta(P_W)$ and $K_W = L_W \cap K$ is a maximal compact subgroup of both $L_W \leq P_W$. 
\item \label{i:Ks} 
$L_\s = Z_\s \cap \theta(Z_\s)$ and $K_\s = L_\s \cap K$ is a maximal compact subgroup of both $L_\s \leq Z_\s$. 
\item \label{i:Kind}
$K_\s$ is independent of our choice of $M \in \sigma$.
\item \label{i:S} 
The stabilizer $\tStab_G(W,F) \cap Z_\s = \tStab_{Z_\s}(F)$ of $F \in B_\s$ is contained in $K_\s$. 
\end{i_list_emph}
\end{lemma}

\begin{proof}
Follow Lemma \ref{L:uniqueK}, let 
\begin{equation}\label{E:hs-phi}
  \fg_\bC \ = \ \op\,\fg^{p,-p}_\varphi
\end{equation}
denote the induced Hodge decomposition on the Lie algebra.  These Hodge structures are polarized by $-\kappa$, where $\kappa$ is the Killing form \cite{MR2918237}.  The Hodge structure is equivalent to a homomorphism $\varphi : \bC^* \to G$ of $\bR$-algebraic groups by $\varphi(z)(v) = z^p\overline{z}{}^q v$ for all $z \in \bC^*$ and $v \in V^{p,q}$ \cite{MR2918237}.  The Cartan involution $\theta : G \to G$ of $K$ is conjugation by the Weil operator $\varphi(\bi) \in G$ \cite[p.~314]{MR0382272}.  The induced Cartan involution $\td\theta:\fg \to \fg$ is $\tAd_{\varphi(\bi)}$.  In particular, the Lie algebra of $K$ is 
\[
  \fk \ = \ \fg_\bR \,\cap\, \op\,\fg_\varphi^{2p,-2p} \,.
\]
The Cartan involution satisfies $\tAd_{\varphi(\bi)}(Y) = -Y$ \cite[Remark 4.21]{MR3505643}; this implies $P_W\,\cap\,\theta(P_W) = L_W$ and yields \emph{\ref{i:KW}}.

Recall the Deligne splitting \eqref{SE:ds-g} of $\fg_\bC = \op\,\fg^{p,q}_{W,F_\tid}$.  Since $\s \subset \fg^{-1,-1}_{W,F_\tid}$, the centralizer $\bfZ_\s \leq \bfP_W$ inherits this decomposition
\begin{equation}\label{E:ds-vs}
  \fz_{\s,\bC} \ = \ 
  \bigoplus_{p+q \le 0} \fz^{p,q}_\s  \,,
  \quad \fz^{p,q}_\s \,=\, \fg^{p,q}_{W,F_\tid}\,\cap\,\fz_{\s,\bC}\,.
\end{equation}
The mixed Hodge structure on $\fg$ is polarized by $\tad_M$.  This implies that 
\begin{equation}\label{E:hs-vs}
  \fl_{\s,\bC} \ = \ \op\,\fz_\s^{p,-p}
\end{equation}
is a pure Hodge structure on $\fl_\s$ that is polarized by $-\kappa$.  The associated Hodge filtration $F_\tid^k(\fl_\s) = \op_{p\ge k}\, \fz_\s^{p,-p}$ is invariant under multiplication by $e^{\bi M} \in \bfG(\bC)$.  It follows that \eqref{E:hs-vs} is a sub-Hodge structure of \eqref{E:hs-phi}.  This yields \emph{\ref{i:Ks}}.  It also implies \emph{\ref{i:Kind}} since \eqref{E:ds-vs} depends only on the mixed Hodge structure $(W,F_\tid)$, and so is independent of $M$.  And \emph{\ref{i:S}} follows from \eqref{SE:S}.
\end{proof}

\begin{corollary} \label{C:fund-set}
Let $\fZ_\s \subset Z_\s$ be a fundamental set for $\bfZ_\s(\bZ)$ satisfying $\fZ_\s\,K_\s = \fZ_\s$ \emph{(as in Definition \ref{dfn:fundset}\ref{i:maxcpt})}.  Then $\fB_\s = \fZ_\s \cdot F_\tid \subset B_\s$ is a fundamental set for the action of $\bfZ_\s(\bZ)$.
\end{corollary}

Our final normalization is the following.

\begin{lemma} \label{L:norm}
Fix fundamental sets $\fB_\s \subset B_\s$ and $\fB_\tau \subset B_\tau$ for the actions of $\bfZ_\s(\bZ)$ and $\bfZ_\tau(\bZ)$, respectively, containing $F_\tid \in \fB_\s \cap \fB_\tau$.  For all $\gamma \in \Gamma_{\s,\tau}$, there exists $\tilde\gamma \in \bfZ_\s(\bZ)\,\gamma\, \bfZ_\tau(\bZ)$ and $F_{\tilde\gamma} \in \fB_\s \cap \tilde\gamma \fB_\tau$ so that both $(\sigma , F_{\tilde\gamma})$ and $(\tau_{\tilde\gamma},F_{\tilde\gamma})$ are nilpotent orbit pairs.
\end{lemma}

\begin{proof}
Without loss of generality $F_\gamma \in \fB_\s$.  Else choose $\z_\gamma \in \bfZ_\s(\bZ)$ so that $\z_\gamma F_\gamma \in \fB_\s$.  (Take $\z_\tid = \tid$.)  Then replace $F_\gamma$ with $\z_\gamma F_\gamma$, and $\gamma$ with $\z_\gamma \,\gamma$.  

Next, note that $\gamma \fB_\tau \ni \gamma  F_\tid$ is a fundamental set of $B_{\tau_\gamma} = \gamma B_\tau$ for the action of $\bfZ_{\tau_\gamma}(\bZ) = \gamma \bfZ_\tau(\bZ) \gamma^{-1}$.  Both $F_\gamma , \gamma F_\tid \in B_{\tau_\gamma}$.  So there exists $g_\gamma \in \bfZ_\tau(\bR)$ so that $F_\gamma = \gamma g_\gamma F_\tid$.  (Take $g_\tid = \tid$.)  And there exists $\x_\gamma \in \bfZ_\tau(\bZ)$ so that $\x_\gamma g_\gamma F_\tid \in \fB_\tau$.  (Take $\x_\tid = \tid$.)  Then $\gamma \x_\gamma\gamma ^{-1} F_\gamma = \gamma \x_\gamma g_\gamma F_\tid \in \gamma \fB_\tau$.  So that $F_\gamma \in \gamma \x_\gamma^{-1} \fB_\tau$.  Replace $\gamma$ with $\gamma \x_\gamma^{-1}$.
\end{proof}

\begin{remark} \label{R:norm}
The upshot of the proof of Lemma \ref{L:norm} is that $\gamma$ has been replaced by $\z_\gamma \gamma \gamma_o^{-1}\x_\gamma^{-1}$ with $\z_\gamma \in \bfZ_\s(\bZ)$ and $\x_\gamma \in \bfZ_\tau(\bZ)$.  Note that:
\begin{i_list}
\item 
The intersection $\s \cap \tau_\gamma = \s \cap \tAd_\gamma(\tau)$ is invariant under these normalizations.
\item
The double coset $[\gamma] \in \bfZ_\s(\bZ) \bs \bfG(\bZ) / \bfZ_\tau(\bZ)$ is preserved by these normalizations.
\end{i_list}
\end{remark}

\subsection{Proof of Theorem \ref{T:doublecoset}} 

Given rational nilpotent cones $\s = \tspan_{\bQ_{>0}}\{N_1,\ldots,N_k\}$ and $\t = \tspan_{\bQ_{>0}}\{N'_1,\ldots,N_l'\}$ (as in Definition \ref{dfn:nilpotentorbit}\ref{i:cone}), define
\[
  \mathfrak{H}_\s^+ \ = \ \left\{\left.\textstyle{\sum_{j=1}^k}\,y_jN_j \ \right| \ y_j \geq 1 \right\}, \
   \mathfrak{H}_\t^+ \ = \ \left\{\left.\textstyle{\sum_{j=1}^l}\,y_jN_j' \ \right| \ y_j \geq 1 \right\}
\]
Moreover, as Lemma \ref{L:norm} we may assume there exists $M\in \mathfrak{H}_\s^+\cap \mathfrak{H}_\t^+$. We set $\varphi=\exp(\bi M)\cdot F_\mathrm{id}$.
\begin{theorem}\label{T:BKTplus}
There exists a finite cover $\mathfrak{H}_\s^+=\bigcup_{1\leq k\leq K}R_k$ such that for each $R_k$, there exists a fundamental set $\fB_{\s,k}\subset B_\s$ as in Lemma \ref{L:norm}, so that the set $\exp(\bi R_k)\cdot\fB_{\s,k}$ is contained in a $\bfP_W(\bZ)$-fundamental set of $P_W\cdot \varphi$.
\end{theorem}

\noindent Assuming Theorem \ref{T:BKTplus} (which is proved in \S\ref{S:prf-BKT+}), we complete the proof of Theorem \ref{T:doublecoset}:  The objective is to show that the image of $\pi(\Gamma_{\s,\tau}) \subset \bfZ_\s(\bZ) \bs \bfG(\bZ) / \bfZ_\tau(\bZ)$ is finite. By \eqref{E:norm} we know $\pi(\Gamma_{\s,\tau}) \subset \bfZ_\s(\bZ) \bs \bfP_W(\bZ) / \bfZ_\tau(\bZ)$. Given $\gamma \in \Gamma_{\s,\tau}$,  Lemma \ref{L:norm} yields $\tilde\gamma \in \bfZ_\s(\bZ)\,\gamma\, \bfZ_\tau(\bZ)$ and $F_{\tilde\gamma} \in \fB_\s \cap \tilde\gamma \fB_\tau$ so that both $(\sigma , F_{\tilde\gamma})$ and $(\tau_{\tilde\gamma},F_{\tilde\gamma})$ are nilpotent orbit pairs.  By Remark \ref{R:norm}, $\s \cap \tau_{\tilde\gamma} = \s \cap \tau_\gamma \not= \emptyset$.  In particular, there exist regions $R_{\s,k}\subset \mathfrak{H}_\s^+, R_{\t,l}\subset \mathfrak{H}_\tau^+$ and corresponding fundamental sets $\fB_{\s,k}\subset B_\s, \fB_{\tau,l}\subset B_\tau$, such that there exists $M\in R_k\cap \mathrm{Ad}_{\tilde{\gamma}}R_l$ and
\[
  \exp(\bi M)\cdot F_{\tilde{\gamma}} \ \subset \ \exp(\bi R_{\s,k})\cdot\fB_{\s,k} \,\cap\, \tilde{\gamma}\exp(\bi R_{\t,l})\cdot\fB_{\t,l} \ \neq \ \emptyset \,.
\]
Theorem \ref{T:BKTplus} thus implies there are only finitely many such $\tilde{\gamma}\in \bfP_W(\bZ)$ (Remark \ref{R:fundset}) for a fixed pair of regions $(R_{\s,k}\subset \mathfrak{H}_\s^+, R_{\t,l}\subset \mathfrak{H}_\tau^+)$.  Running over all (finitely many) such pairs of regions, we conclude the image $\pi(\Gamma_{\s,\tau}) \subset \bfZ_\s(\bZ) \bs \bfG(\bZ) / \bfZ_\tau(\bZ)$ must be finite.  \hfill \qed

\section{Constructing weak fan by finite subdivisions}\label{S:weakfan}

We now turn to the proof of Theorem \ref{T:main3}, as outlined in Remark \ref{R:main3}.  We begin with some properties of the cardinality $|\cI_{\s,\tau}| < \infty$ (Corollary \ref{C:Ifinite}), that will be needed for the finite subdivision to a weak fan (in \S\ref{S:finite-subdiv}).

\subsection{Counting intersections}

The cardinality $|\cI_{\s,\tau}|$ depends only on the $\Gamma$--conjugacy classes of $\s$ and $\tau$.  

\begin{lemma} \label{L:card}
Given rational nilpotent cones $\s,\tau$, we have $|\cI_{\s,\tau}| = |\cI_{\tAd_\gamma(\s),\tAd_{\gamma'}(\tau)}|$ for all $\gamma,\gamma' \in \Gamma$.
\end{lemma}  

\noindent Lemma \ref{L:card} is a consequence of the Lemma \ref{L:symmetries}.  The latter may be verified directly from the definition \eqref{SE:Ist} of $\Gamma_{\s,\tau}$.

\begin{lemma}\label{L:symmetries}
\begin{i_list_emph}
\item \label{i:s1}
The map $\gamma \mapsto \gamma^{-1}$ defines a bijection $\Gamma_{\s,\tau} \to \Gamma_{\tau,\s}$.  And the map $\s \cap \tau_\gamma \mapsto \tau \cap \s_{\gamma^{-1}} = \tAd_\gamma^{-1}(\s\cap\tau_\gamma)$ is a bijection $\cI_{\s,\tau} \to \cI_{\tau,\s}$.
\item \label{i:s2}
If $\b \in \Gamma$, then $\gamma \mapsto \gamma \b$ defines a bijection $\Gamma_{\s,\tAd_\b(\tau)} \to \Gamma_{\s,\tau}$.  We have $\cI_{\s,\tau} = \cI_{\s,\tAd_\b(\tau)}$ for all $\b \in \Gamma$.
\end{i_list_emph}
\end{lemma}

\noindent The last property of the cardinality that we will make use of is

\begin{lemma}\label{L:containment}
Suppose that $\s_1,\s_2$ are rational nilpotent cones.  If $\s_1 \subset \s_2$ and $\tdim\,\s_1 = \tdim\,\s_2$, then $\bfZ_{\s_1} = \bfZ_{\s_2}$.  If both cones can be completed to nilpotent orbit pairs $(\s_j,F_j)$, then $\Gamma_{\sigma_1,\tau} \subset \Gamma_{\sigma_2,\tau}$ for any rational nilpotent cone $\tau$, and $\pi(\Gamma_{\sigma_1,\tau}) \subset \pi(\Gamma_{\sigma_2,\tau})$.  The map $\cI_{\s_1,\tau} \to \cI_{\s_2,\tau}$ sending $\s_1 \cap \tau_\gamma \mapsto \s_2 \cap \tau_\gamma$ is injective.
\end{lemma}

\noindent The proof of Lemma \ref{L:containment} will make use of the following result.

\begin{lemma}\label{L:eqdim}
Suppose that $\s_1,\s_2$ are rational nilpotent cones that can be completed to nilpotent orbit pairs $(\s_1,F_1)$ and $(\s_2,F_2)$, respectively.  If $\s_1 \cap \s_2 \not=\emptyset$ is open in both $\s_1$ and $\s_2$, then $B_{\s_1} = B_{\s_2}$.
\end{lemma}

\begin{proof}[Proof Lemma \ref{L:containment}]
The assertion $\bfZ_{\s_1} = \bfZ_{\s_2}$ is immediate.  Next the existence of nilpotent orbit pairs $(\s_j,F_j)$ implies that $(\s_1,F)$ is a nilpotent orbit pair if and only if $(\s_2,F)$ is a nilpotent orbit pair (Lemma \ref{L:eqdim}). Thus $\Gamma_{\sigma_1,\tau} \subset \Gamma_{\sigma_2,\tau}$, and $\pi(\Gamma_{\sigma_1,\tau}) \subset \pi(\Gamma_{\sigma_2,\tau})$ follows.

For the injectivity statement, assume $\gamma,\gamma' \in \Gamma_{\s_1,\tau}$ and that $\s_1 \cap \tau_\gamma \not = \s_1 \cap \tau_{\gamma'} \in \cI_{\s_1,\tau}$.  Then, without loss of generality, there exists $N \in \s_1 \cap \tau_\gamma$ so that $N \not \in \s_1 \cap \tau_{\gamma'}$.  Since $\s_1 \subset \s_2$ this implies $N \in \s_2 \cap \tau_\gamma$, but $N \not\in \tau_{\gamma'}$.  Thus $\s_2 \cap \tau_\gamma \not = \s_2 \cap \tau_{\gamma'}$, and the map is injective.
\end{proof}

\noindent The proof of Lemma \ref{L:eqdim} will make use of the following result.


\begin{proof}[Proof of Lemma \ref{L:eqdim}]
Without loss of generality, both $(\s_j,F_j)$ are $\bR$--split, cf.~\cite[(2.20)]{MR840721}.  Let $\fg_\bC = \oplus\,\fg^{p,q}$ be the Deligne splitting \eqref{SE:ds-g} induced by the nilpotent orbit pair $(\s_1,F_1)$.  A priori $\s_1 \subset \fg^{-1,-1}$.  The hypothesis that $\tspan_\bQ \s_1 = \tspan_\bQ \s_2$ implies that $\s_2 \subset \fg^{-1,-1}$, as well.

Since $\s_1 \cap \s_2 \not= \emptyset$, we necessarily have $W(\s_1) = W(\s_2)$, \cite[Theorem 3.3]{MR664326}.  Set $W = W(\s_j)$.  The real Lie group $\tStab_G(W,F_1)$ stabilizing both the weight filtration and $F_1$ preserves the Deligne splitting; in particular, $\tStab_G(W,F_1)$ acts on $\fg_\bR \cap \fg^{-1,-1}$ (via the adjoint representation). The assumption that $\s_1 \cap \s_2$ is nonempty implies that both cones $\s_1$ and $\s_2$ are contained in the same orbit of this action; cf.~\cite[Lemma 3.5]{MR3751291}, where $M^0_\bR$ denotes the connected identity component of $\tStab_G(W,F_1)$.  This implies that $(\s_2,F_1)$ is also a nilpotent orbit pair; that is $F_1 \in B_{\s_2}$.  This proves that $B_{\s_1} \subset B_{\s_2}$.  The same argument goes though with the roles of $1$ and $2$ swapped.
\end{proof}

\begin{lemma} \label{L:weakfanI}
Let $\Sigma$ be a collection of rational nilpotent cones \emph{(as in Definition \ref{dfn:nilpotentorbit}\ref{i:cone})} with the following properties:
\begin{i_list_emph}
\item
The collection is closed under taking faces: if $\s \in \Sigma$, and $\tau$ is a face of $\s$, then $\tau \in \Sigma$.
\item
The collection is closed under the action of $\Gamma$: if $\sigma \in \Sigma$ and $\gamma \in \Gamma$, then $\tAd_\gamma(\s) \in \Sigma$.
\item
The collection contains only finitely many $\tAd_\Gamma$--orbits. 
\item 
Every cone $\s \in \Sigma$ can be completed to a nilpotent orbit pair $(\s,F)$.
\end{i_list_emph}
Then $\Sigma$ is a weak fan if and only if $|\cI_{\s,\s}|=1$ for all $\s \in \Sigma$, and $|\cI_{\s,\tau}|=0$ for all distinct orbits $\tAd_\Gamma(\s),\tAd_\Gamma(\tau) \subset \Sigma$.
\end{lemma}

\begin{proof}
This follows directly from Definition \ref{dfn:weakfan} and \eqref{SE:Ist}.
\end{proof}

\subsection{Finite subdivision of two-dimensional cones} \label{S:finite-subdiv}

Let $\Sigma$ be any collection of rational nilpotent cones satisfying the properties (i-iv) of Lemma \ref{L:weakfanI}, and the additional property that $\tdim\,\s\le 2$ for all $\s \in \Sigma$.  We will show that $\Sigma$ admits a finite subdivision into a weak fan (Proposition \ref{P:wf}).

\subsubsection{First modification of $\Sigma$} 

Let $\tAd_\Gamma(\s_1),\ldots, \tAd_\Gamma(\s_\ell) \subset \Sigma$ denote the distinct $\tAd_\Gamma$-orbits of two-dimensional cones in $\Sigma$.  Since $\Sigma$ consists of only finitely many $\tAd_\Gamma$-orbits, Corollary \ref{C:Ifinite} and Lemma \ref{L:symmetries}\emph{\ref{i:s2}} imply that 
\[
  \s_j \ - \ \{ \s_j \cap \tau_\gamma \ | \ 
  \tau \in \Sigma \,,\ 
  \gamma \in \Gamma_{\s_j,\tau} \,,\ 
  \tdim\,(\s_j \cap \tau_\gamma)=1 \}
\]
is a finite disjoint union of two-dimensional cones $\varrho_{j,a}$.  Let $\tAd_\Gamma(\varrho_1), \ldots,\tAd_\Gamma(\varrho_k)$ be the distinct $\Gamma$-orbits of all two-dimensional cones that we obtain in this way.

Let $\Sigma'$ be the collection obtained from $\Sigma$ by: (i) replacing the orbits $\tAd_\Gamma(\s_j)$ with the orbits $\tAd_\Gamma(\varrho_a)$; and (ii) adding the (distinct) $\Gamma$-orbits of the one-dimensional $\s_j \cap \tAd_\gamma(\s_i)$ with $\gamma \in \Gamma_{\s_j,\s_i}$ (that are not already in $\Sigma$).  This ensures that $\Sigma'$ is closed under taking faces, and that each $\s_i \cap \tAd_\gamma(\s_j)$ is a disjoint union of elements of $\Sigma'$.

The collection $\Sigma'$ satisfies the properties (i-iv) of Lemma \ref{L:weakfanI}.  

\begin{lemma} \label{L:mod1}
We have $|\cI_{\varrho_a,\varrho_b}| = \d_{ab}$, for all $1 \le a,b \le k$.
\end{lemma}

\begin{proof}
Suppose that $\gamma \in \Gamma_{\varrho_a , \varrho_b}$.  By definition of the $\varrho_a$, there exist $\s_i \supset \varrho_a$ and $\s_j \supset \varrho_b$.  Then Lemmas \ref{L:symmetries}\emph{\ref{i:s1}} and \ref{L:containment} imply $\gamma \in \Gamma_{\s_i,\s_j}$.  By construction of the $\varrho_a$ this forces $\varrho_a = \tAd_\gamma(\varrho_b)$ or $\varrho_a \cap \tAd_\gamma(\varrho_b) = \emptyset$.  Thus $|\cI_{\varrho_a,\varrho_b}| = \d_{ab}$.
\end{proof}

\subsubsection{Second modification of $\Sigma$} 

Let $\tAd_\Gamma(\tau_1) , \ldots , \tAd_\Gamma(\tau_m)$ denote the distinct $\Gamma$--orbits of one-dimensional cones in $\Sigma'$.  The set of one-dimensional cones 
\[
  \{ \tAd_\gamma(\tau_j) \ | \ 
  \gamma \in \Gamma_{\varrho_a,\tau_j} \,,\ j=1,\ldots,m\} 
  \ \subset \ \varrho_a
\]
is finite by Corollary \ref{C:Ifinite}.  The complement in $\varrho_a$ is a finite, disjoint union of two-dimensional rational nilpotent cones $\{\varsigma_{a,i}\}$.  Let $\tAd_\Gamma(\varsigma_1), \ldots,\tAd_\Gamma(\varsigma_k)$ be the distinct $\Gamma$-orbits of all two-dimensional cones that we obtain in this way.  Let $\Sigma''$ be the collection obtained from $\Sigma'$ by replacing the orbits $\tAd_\Gamma(\varrho_a)$ with the orbits $\tAd_\Gamma(\varsigma_i)$.  

The collection $\Sigma''$ also satisfies the properties (i-iv) of Lemma \ref{L:weakfanI}.

\begin{proposition}\label{P:wf}
The collection $\Sigma''$ is a weak fan.
\end{proposition}

\begin{proof}
Suppose that $\gamma \in \Gamma_{\varsigma_i,\varsigma_j}$.  By definition of the $\varsigma_i$, there exist $\varrho_a \supset \varsigma_i$ and $\varrho_b \supset \varsigma_j$.  Then Lemmas \ref{L:symmetries}\emph{\ref{i:s1}} and \ref{L:containment} imply $\gamma \in \Gamma_{\varrho_a,\varrho_b}$.  Lemma \ref{L:mod1} implies $a=b$.  By construction of the $\varsigma_i$ this forces $\varsigma_i = \tAd_\gamma(\varsigma_j)$ or $\varsigma_i \cap \tAd_\gamma(\varsigma_j) = \emptyset$.  Thus $|\cI_{\varsigma_i,\varsigma_j}| = \d_{ij}$.

Suppose that $\gamma \in \Gamma_{\varsigma_i,\tau_j}$.  Again, there exists $\varrho_a\supset \varsigma_i$, and Lemma \ref{L:containment} implies $\gamma \subset \Gamma_{\varrho_a,\tau_j}$.  The construction of the $\varsigma_{a,i}$ as the disjoint union
\[
  \bigcup_i \,\varsigma_{a,i} \ = \ \varrho_a
  \ - \ \bigcup_j \{ \tAd_\gamma(\tau_j) \ | \ 
  \gamma \in \Gamma_{\varrho_a,\tau_j} \}
\]
implies $\varsigma_i \cap \tAd_\gamma(\tau_j) = \emptyset$.  Thus $|\cI_{\varsigma_i,\tau_j}|=0$.

Since the one-dimensional cones necessarily satisfy $|\cI_{\tau_i,\tau_j}| = \d_{ij}$, the proposition now follows from Lemmas \ref{L:card} and \ref{L:weakfanI}.
\end{proof}

\subsection{Proof of Theorem \ref{T:main3}} \label{S:prf-main3}

Assume $\tdim\,S=2$, and let $\tilde S \supset S$ be some smooth projective completion with simple normal crossing divisor $\tilde S \bs S$.  Let $S_1 , \ldots , S_\ell$ denote the smooth irreducible components of $\partial S$.  Set $S^*_{ij} = S_i \cap S_j$, and $S_i^* = S_i \bs \cup_j S^*_{ij}$.  Given $I \subset \{1,\ldots,\ell\}$ with $|I| \in \{1,2\}$, Schmid's nilpotent orbit theorem associates to each (connected component of) $S^*_I$ a $\Gamma$--conjugacy class $\tAd_\Gamma(\s_I)$ of rational nilpotent cones.  The nilpotent cones associated to the codimension one strata $S_i^*$ are the $\s_i = \tspan_{\bQ_{>0}}\{N_i\}$.  The nilpotent cones associated to the codimension two strata $S^*_{ij}$ are the $\s_{ij} = \tspan_{\bQ_{>0}}\{N_i,N_j\}$.  

Let $\Sigma_{\Phi,\tilde S}$ be the set of all $\tAd_\gamma(\s_I)$, with $\gamma \in \Gamma$ and $I \subset \{1,\ldots,\ell\}$.  The collection $\Sigma_{\Phi,\tilde S}$ satisfies the properties (i-iv) of Lemma \ref{L:weakfanI}.  Let $\Sigma = \Sigma''_{\Phi,\tilde S}$ be the weak fan obtained from $\Sigma_{\Phi,\tilde S}$, as in Proposition \ref{P:wf}.  Then Theorem \ref{T:main3} follows from Lemma \ref{L:bar-S}.  \hfill\qed

\begin{lemma} \label{L:bar-S}
There exists a smooth completion $\olS \supset S$ with simple normal crossing divisor $\olS \bs S$ and a birational morphism $\olS \to \tilde S$, given by a composition of blow-ups, so that for each nilpotent orbit pair $(\s,F)$ with $\s \in \Sigma_{\Phi,\olS}$, there exists a unique minimal $\tau \in \Sigma = \Sigma''_{\Phi,\tilde S}$ so that $\s \subset \tau$ and $(\tau,F)$ is a nilpotent orbit pair.
\end{lemma}

\begin{proof}
Consider the blow-up $B_{ij}(\tilde S)$ of $\tilde S$ at (a point $p$ in the) codimension two $S_{ij}$.  The lemma follows directly from the fact that the one-dimensional nilpotent cone associated to the exceptional divisor $E = \bP^1$ is $\s_E = \tspan_{\bQ_{>0}}\{N_i + N_j\}$; this is pictured in Figure \ref{fig:bu1}.
\begin{figure}
\caption{Monodromy after blow-up at $p \in S_{ij}^*$}
\smallskip
\begin{tikzpicture}
  \draw (0,1) -- (3,1);
  \node[right] at (3,1) {$S_i$};
  \draw (1.85,0.95) to [out=-90,in=180] (2,0.8);
  \draw (2,0.8) to [out=0,in=270] (2.15,1);
  \draw (2.15,1) to [out=90,in=0] (2,1.2);
  \draw[->] (2,1.2) to [out=180,in=90] (1.85,1.05);
  \node[below] at (2,0.8) {$N_i$};
  \draw[red] (1,0) -- (1,4);
  \node[above,left,red] at (1,3.8) {$E$};
  \draw[red] (0.95,1.85) to [out=180,in=270] (0.8,2);
  \draw[red] (0.8,2) to [out=90,in=180] (1,2.15);
  \draw[red] (1,2.15) to [out=0,in=90] (1.2,2);
  \draw[->,red] (1.2,2) to [out=270,in=0] (1.05,1.85);
  \node[left,red] at (0.8,2) {$N_i+N_j$};
  \draw (0,3) -- (3,3);
  \node[right] at (3,3) {$S_j$};
  \draw (1.85,2.95) to [out=-90,in=180] (2,2.8);
  \draw (2,2.8) to [out=0,in=270] (2.15,3);
  \draw (2.15,3) to [out=90,in=0] (2,3.2);
  \draw[->] (2,3.2) to [out=180,in=90] (1.85,3.05);
  \node[below] at (2,2.8) {$N_j$};
  \draw[->] (5,2) -- (7,2);
  \node[above] at (6,2) {$B_{ij}(\tilde S)$};
\end{tikzpicture}
\qquad
\begin{tikzpicture}
  \draw (0,1) -- (3,1);
  \node[right] at (3,1) {$S_i$};
  \draw (1.85,0.95) to [out=-90,in=180] (2,0.8);
  \draw (2,0.8) to [out=0,in=270] (2.15,1);
  \draw (2.15,1) to [out=90,in=0] (2,1.2);
  \draw[->] (2,1.2) to [out=180,in=90] (1.85,1.05);
  \node[below] at (2,0.8) {$N_i$};
  \draw (1,0) -- (1,3);
  \node[above] at (1,3) {$S_j$};
  \draw (0.95,1.85) to [out=180,in=270] (0.8,2);
  \draw (0.8,2) to [out=90,in=180] (1,2.15);
  \draw (1,2.15) to [out=0,in=90] (1.2,2);
  \draw[->] (1.2,2) to [out=270,in=0] (1.05,1.85);
  \node[left] at (0.8,2) {$N_j$};
  \fill[red] (1,1) circle (0.1);
\end{tikzpicture}
\label{fig:bu1}
\end{figure}
After a finite number of such blow-ups of we obtain the desired $\olS$.
\end{proof}  

\begin{example}
As an illustrative example, suppose that the process of obtaining $\Sigma''_{\Phi,\tilde S}$ subdivides $\s_{ij} = \tspan_{\bQ_{>0}}\{N_i,N_j\} \in \Sigma_{\Phi,\tilde S}$ into two cones $\tau = \tspan_{\bQ_{>0}}\{ N_j , 3 N_i + N_j \} \in \Sigma''_{\Phi,\tilde S}$ and $\tau' = \tspan_{\bQ_{>0}}\{ 3 N_i + N_j , N_i \} \in \Sigma''_{\Phi,\tilde S}$, as pictured in Figure \ref{fig:tau}.
\begin{figure}
\caption{Finite subdivision of $\s_{ij} = \tspan_{\bQ_{>0}}\{N_i,N_j\}$}
\smallskip
\begin{tikzpicture}
  \draw (0,0) -- (2,0);
  \node[right] at (2,0) {$N_i$};
  \draw (0,0) -- (0,2);
  \node[above] at (0,2) {$N_j$};
  \draw[orange] (0,0) -- (6,2); 
  \node[right,orange] at (6,2) {$3N_i + N_j$};
  \draw[red] (0,0) -- (2,2);
  \node[above,red] at (2,2) {$N_1+N_2$};
  \draw[blue] (0,0) -- (4,2);
  \node[above,blue] at (4,2) {$2N_1+N_2$};
\end{tikzpicture}
\label{fig:tau}
\end{figure}
Taking a sequence of three blow-ups over $p$, as pictured in Figure \ref{fig:bu2},
\begin{figure}
\caption{Blow-ups for subdivision of $\s_{ij}$}
\smallskip
\begin{tikzpicture}
  \draw (0,4) -- (3,4);
  \node [left] at (0,4) {$S_j$};
  \draw (1.85,3.95) to [out=-90,in=180] (2,3.8);
  \draw (2,3.8) to [out=0,in=270] (2.15,4);
  \draw (2.15,4) to [out=90,in=0] (2,4.2);
  \draw[->] (2,4.2) to [out=180,in=90] (1.85,4.05);
  \node [above] at (2,4.2) {$N_j$};
  \draw[red] (1,5) -- (1,1);
  \node [above,red] at (1,5) {$E_1$};
  \draw[red] (0.95,2.85) to [out=180,in=270] (0.8,3);
  \draw[red] (0.8,3) to [out=90,in=180] (1,3.15);
  \draw[red] (1,3.15) to [out=0,in=90] (1.2,3);
  \draw[->,red] (1.2,3) to [out=270,in=0] (1.05,2.85);
  \node[left,red] at (0.8,3) {$N_i+N_j$};
  \draw[blue] (0,2) -- (4,2);
  \node [left,blue] at (0,2) {$E_2$};
  \draw[blue] (1.85,1.95) to [out=-90,in=180] (2,1.8);
  \draw[blue] (2,1.8) to [out=0,in=270] (2.15,2);
  \draw[blue] (2.15,2) to [out=90,in=0] (2,2.2);
  \draw[->,blue] (2,2.2) to [out=180,in=90] (1.85,2.05);
  \node [below,blue] at (2,1.8) {$2N_i+N_j$};
  \draw[orange] (3,3) -- (3,-1);
  \node [below,orange] at (3,-1) {$E_3$};
  \draw[orange] (2.95,0.85) to [out=180,in=270] (2.8,1);
  \draw[orange] (2.8,1) to [out=90,in=180] (3,1.15);
  \draw[orange] (3,1.15) to [out=0,in=90] (3.2,1);
  \draw[->,orange] (3.2,1) to [out=270,in=0] (3.05,0.85);
  \node[right,orange] at (3.2,1) {$3N_i+N_j$};
  \draw (2,0) -- (5,0);
  \node [right] at (5,0) {$S_i$};
  \draw (3.85,-0.05) to [out=-90,in=180] (4,-0.2);
  \draw (4,-0.2) to [out=0,in=270] (4.15,0);
  \draw (4.15,0) to [out=90,in=0] (4,0.2);
  \draw[->] (4,0.2) to [out=180,in=90] (3.85,0.05);
  \node [below] at (4,-0.2) {$N_i$};
\end{tikzpicture}
\label{fig:bu2}
\end{figure}
replaces the cone $\s_{ij} \in \Sigma_{\Phi,\tilde S}$ with the cones $\s_{ij}^1 = \tspan_{\bQ_{>0}}\{N_i+N_j,N_j\} \subset \tau$, $\s_{ij}^2 = \tspan_{\bQ_{>0}}\{2N_i+N_j,N_i+N_j\} \subset \tau$, $\s_{ij}^3 = \tspan_{\bQ_{>0}}\{3N_i+N_j,2N_i+N_j\} \subset \tau$, and $\s_{ij}^4 = \tspan_{\bQ_{>0}}\{3N_i+N_j,N_i\} = \tau'$.  These new $\s_{ij}^a \in \Sigma_{\Phi,\olS}$.
\end{example}

Let $\olS \to \hat S$ be the birational morphism obtained by blowing-down the exceptional divisors $E \subset \olS$ with $\s_E \not\in \Sigma = \Sigma''_{\Phi,\tilde S}$.  In other words, $\hat{S}\rightarrow \tilde{S}$ is the weighted blow-up (the log modifiction) induced by the weak fan $\Sigma$, and  $\olS\rightarrow \hat{S}$ is a resolution of singularities. 

\begin{corollary}\label{C:prf}
There is a commutative diagram
\begin{equation}\label{F:Res}
\begin{tikzcd}
\olS \arrow[d] \arrow[r, "\Phi_\Sigma"] & \Gamma \backslash D_{\Sigma} \\
\hat S \arrow[ru,"\hat\Phi_\Sigma"'] & 
\end{tikzcd}
\end{equation}
of morphisms of so that $\Phi_\Sigma(\olS) = \hat\Phi_\Sigma(\hat S)$.
\end{corollary}

\begin{proof}
Fix $p\in \hat{S}$. For any $p'\in \olS$ lying above $p$, any neighborhood $U'$ of $p'\in\olS$ is mapped to a neighborhood $U$ of $p\in \tilde S$, and this map is an isomorphism $U' \cap S \xrightarrow{\sim} U\cap S$. Since the nilpotent orbits are determined by $\left.\Phi\right|_{U'\cap S} = \left.\Phi\right|_{U \cap S}$, this implies 
that the $\Gamma$--conjugacy class of the orbit $\exp(\bC\sigma_{p})F_p$ is constant along the fibres of $\olS \to \hat S$. Since by definition, $\Phi_{\Sigma}(p')$ is the nilpotent orbit pair $(\sigma_p, F_{p'})$ mod $\Gamma$, $\Phi_\Sigma : \olS \to \Gamma \bs D_\Sigma$ is locally constant along the exceptional divisors of $\olS \to \tilde S$. This, along with Remark \ref{R:KNU} and Lemma \ref{L:bar-S}, yields the corollary.
\end{proof}

\begin{remark} \label{R:prf}
Note that $\hat S$ is the minimal projective completion of $S$ through which both $\olS \to \tilde S$ and $\Phi_\Sigma : \olS \to \Gamma \bs D_\Sigma$ factor.  However, unlike $\olS$, $\hat S$ need not be smooth, nor $\hat S \bs S$ a normal crossing divisor.
\end{remark}

\begin{figure}
\caption{Extension of $\Phi$: schematic of the constructions}
\smallskip
\begin{tikzcd}[column sep=small,row sep=small]
  & \Gamma \bs D \arrow[dr,hook] & \\
  S \arrow[ur,"\Phi"] \arrow[r,hook] 
  	\arrow[rd,hook] \arrow[rdd,hook]
  & \olS \arrow[r,"\Phi_\Sigma"] \arrow[d]
  & \Gamma \bs D_\Sigma \\
  & \hat S \arrow[d] \arrow[ru,"\hat\Phi_\Sigma"'] & \\
  & \tilde S &
\end{tikzcd}
\end{figure}

\begin{remark}[The hermitian case]
The construction above provides a finite subdivision of $\Sigma_{\Phi,\tilde S}$ into a \emph{weak fan} $\Sigma_{\Phi,\tilde S}''$ under the hypothesis that $\tdim\,S=2$.  If $D$ is hermitian, then we may both: (i) drop the constraint on $\tdim\,S$, and (ii) obtain a \emph{fan} by finite subdivision of $\Sigma_{\Phi,\tilde S}$.  The essential observation here is that given $\s,\tau \in \Sigma_{\Phi,\tilde S}$, we have
\begin{equation}\label{E:hermitian}
  \Gamma_{\s,\tau} \ = \ \{
  \gamma \in \Gamma \ | \ 
  \s \cap \tau_\gamma \not= \emptyset\} \,.
\end{equation}
Assuming \eqref{E:hermitian} for the moment, Corollary \ref{C:Ifinite} implies that $\cI_{\s,\tau} = \{ \s \cap \tau_\gamma \ | \ \gamma \in \Gamma\}$ is finite.  We may then obtain a fan $\Sigma$ from $\Sigma_{\Phi,\tilde S}$ by finite subdivision, as done in \cite[\S7]{deng2023}.  Then it is natural to ask Theorems \ref{T:main1} and \ref{T:main3} hold when we replace the hypothesis $\tdim\,S=2$ with the condition that $D$ be hermitian.  To prove this, it suffices to establish the corresponding analog of Lemma \ref{L:bar-S}.  We anticipate that this can be done, but have elected not to undertake it here.
\end{remark}

\begin{proof}[Proof of \eqref{E:hermitian}]
The containment $\Gamma_{\s,\tau} \subset \{
\gamma \in \Gamma \ | \ \s \cap \tau_\gamma \not= \emptyset\}$
is by definition \eqref{SE:Ist}.  For the converse, assume that $\s\cap\tau \not=\emptyset$.  We must show that there exists $F \in \check D$ so that both $(\s,F)$ and $(\tau,F)$ are nilpotent orbit pairs.

By definition of $\Sigma_{\Phi,\tilde S}$ there exists $F \in \check D$ so that $(\s,F)$ is nilpotent orbit pair.  Without loss of generality, we may assume the associated mixed Hodge structure is $\bR$--split, as in \S\ref{S:nI}.  Let $\fg_\bC = \op\,\fg^{p,q}_{W,F}$ be the Deligne splitting induced by the mixed Hodge structure $(W(\s)[-\sfn],F)$.  The salient feature of the Hermitian case is that 
\[
  W_{-2}(\fg_\bC) \ = \ \fg^{-1,-1}_{W,F} \,.
\]
We always have $\s \subset W_{-2}(\fg_\bQ)$ by definition of $W(\s)$.  Then $\sigma \cap \tau\not=\emptyset$, and the independence of the weight filtration on $N \in \tau$, implies $\tau \subset W_{-2}(\fg_\bC) = \fg^{-1,-1}_{W,F}$.  We claim that this implies $(\tau,F)$ is a nilpotent orbit pair; equivalently, $\tau$ polarizes the mixed Hodge structure $(W(\s)[-\sfn],F)$.

To see this, let $G^0$ be the connected identity component of the subgroup of $G$ preserving the Deligne splitting of $\fg_\bC$ (under the adjoint action).  Fix $N \in \sigma \cap \tau$.  The orbit $\cO = \tAd(G^0)(N) \subset \fg_\bR \cap \fg^{-1,-1}_{W,F}$ is a connected component of the set of all $N' \in \fg_\bR \cap \fg^{-1,-1}_{W,F}$ with $W(\s) = W(N')$, \cite[Lemma 3.5]{MR3751291}.  So $\sigma,\tau \subset \cO$.  And because $G^0$ preserves $F$, $W$ and $Q$, any $N' \in \cO$ polarizes the mixed Hodge structure $(W(\s)[-\sfn],F)$; in particular, $\tau \subset \cO$ polarizes the mixed Hodge structure.
\end{proof}

\section{The case of mixed period maps}\label{S:VMHS}

Kato--Nakayama--Usui's theory \cite{MR3084721} applies not just to period maps (of pure, polarized Hodge structures), but more generally to mixed period maps
\begin{equation}\label{E:Psi}
  \Psi : S \to \Gamma \bs \cM
\end{equation}
induced by admissible, graded-polarized variations of mixed Hodge structure.  (In this section we assume the reader is familiar with mixed Hodge structures, and will be correspondingly brief.  Good references include \cite{MR2393625, MR866665} and \cite[\S3, \S8]{MR3288678}.)  And the proof of Theorem \ref{T:main1} applies in this more general setting to yield

\begin{theorem}\label{T:main1-mixed}
Suppose that $\tdim\,S=2$.  There exists a smooth projective completion $\olS \supset S$, with simple normal crossing divisor $\partial S = \olS \bs S$, and a logarithmic manifold $\Gamma \bs \cM_\Sigma$ so that $\Gamma \bs \cM \inj \Gamma \bs \cM_\Sigma$ and the mixed period map \eqref{E:Psi} extends to a morphism $\Psi_\Sigma : \olS \to \Gamma \bs \cM_\Sigma$ of logarithmic manifolds.  The image $\Psi_\Sigma(\olS)$ is a compact algebraic space. 
\end{theorem}

\begin{proof}
The argument is almost identical to that establishing Theorem \ref{T:main1}.  The key points to note are: (i) Kato--Nakayama--Usui's Theorem \ref{T:KNU} holds in the more general setting of mixed period maps, cf.~\cite{MR2721860, MR3084721}.  So the crux of the matter is to exhibit a weak fan for the mixed period map.  By \cite[Theorem 2.1 \& Lemma 4.2]{MR2721860} it suffices to exhibit a weak fan for the induced variation of pure Hodge structures on the weight graded quotients.  The existence of this fan is provided by Theorem \ref{T:main3}.
\end{proof}

We apply Theorem \ref{T:main1-mixed} to the construction of N\'eron models. Here we assume that the reader is familiar with intermediate Jacobians, normal functions and (generalized) N\'eron models, and will be commensurately brief in our discussion.  Briefly and informally, given a family of intermediate Jacobians $\cJ \to S$ and an admissible normal function $\nu : S \to \cJ$, a N\'eron model should provide an extension $\cJ_\mathrm{e} \to \olS$ of the family to which the normal function extends $\nu_\mathrm{e} : \olS \to \cJ_\mathrm{e}$.  Such constructions are motivated by their applications to the study of families of cycles in general, and the Hodge conjecture in particular \cite{MR2534102, MR2385303, MR2313336, MR2796415}.  In the case that $\tdim\,S=1$, two such constructions have been given by Kato--Nakayama--Usui \cite{MR3084721} and Green--Griffiths--Kerr \cite{MR2601630}.  The two constructions are homeomorphic \cite{MR2832807}, and Theorem \ref{T:neron} below may be viewed as an extension of this work to the case $\tdim\,S=2$ via the approach of Kato--Nakayama--Usui.  A general construction (over an arbitrary base), exhibiting the expected properties of the the identity component of the N\'eron model, has been given by Schnell \cite{MR2897692} using Saito's theory of mixed Hodge modules.

Suppose that $\mathbb{V} \to S$ is a polarized variation of integral Hodge structures (as in \S\ref{S:mainresult}) of weight $\mathsf{w} = -1$.  To each point $s \in S$ we associate a complex torus $J_s = V_\bC /(F^0_s(V_\bC) + V_\bZ)$.  In the geometric case that $\mathbb{V} \subset R^{2\mathsf{n}-1}f_*\bZ(\mathsf{n})/(\mathrm{torsion})$, for a smooth proper morphism $f : \cX \to S$, $J_s$ is the $\mathsf{n}$-th intermediate Jacobian of $X_s = f^{-1}(s)$.  In general we have a natural identification of $J_s$ with the set $\mathrm{Ext}^1(\bZ(0),\mathbb{V}_s)$ of isomorphism classes of short exact sequences $0 \to \mathbb{V}_s \to V' \to \bZ(0) \to 0$ of mixed Hodge structures.  This provides the following description of the bundle $\cJ = \cup_s J_s$ over $S$.  

Recall the notation $V_\bZ = \mathbb{V}_{s_o}$ of \S\ref{S:mainresult}.  Fix a lattice $V_\bZ'$ extending $V_\bZ$ by short exact sequence $0 \to V_\bZ \to V'_\bZ \to \bZ \to 0$.  Let $\cM$ be the mixed period domain parameterizing graded-polarized mixed Hodge structures $V'_\bZ$ with weight filtration $W_{-2}=0$, $W_{-1} \simeq V_\bZ$ and $W_0 = V_\bZ'$.  We have a natural quotient map $\cM \to D$.  Let $\tAut(V'_\bZ,W)$ be the automorphisms of $V'_\bZ$ preserving the weight filtration, and let $\Gamma' \simeq \Gamma \ltimes V_\bZ$ be the preimage of $\Gamma \times \{\tId\}$ under the natural projection $\tAut(V_\bZ',W) \to \tAut(V_\bZ) \times \tAut(\tGr^W_0)$.  (Equivalently, $\Gamma' \subset \tAut(V_\bZ',W)$ is the subgroup of all elements acting trivially on $\tGr^W_0$, and as $\Gamma$ on $W_{-1}$.)  Then we have induced projections $\Gamma'\bs \cM \to \Gamma \bs D$, and $\cJ$ may be realized as the fibre product \cite{MR0771384}
\begin{equation}\label{E:fp*} \begin{tikzcd}
  \cJ \arrow[r] \arrow[d] & \Gamma'\bs \cM \arrow[d,"\pi"] \\
  S \arrow[r,"\Phi"] & \Gamma \bs D \,.
\end{tikzcd} \end{equation}

\begin{theorem}\label{T:neron}
Let $\Phi : S \to \Gamma \bs D$, $\olS$ and $\Sigma$ be as in Theorem \ref{T:main3}.  There is a morphism $\pi_{\Sigma'} : \Gamma'\bs \cM_{\Sigma'} \to \Gamma \bs D_\Sigma$ of log manifolds extending the map $\pi : \Gamma' \bs \cM \to \Gamma \bs D$ of \eqref{E:fp*}.  The associated fibre product 
\begin{equation}\label{E:fp}
\begin{tikzcd}
  \cJ_{\Sigma'} \arrow[r] \arrow[d] & \Gamma'\bs \cM_{\Sigma'} 
  \arrow[d] \\
  \olS \arrow[r,"\Phi"] & \Gamma \bs D_\Sigma \,,
\end{tikzcd} \end{equation}
which may be viewed as a completion of \eqref{E:fp*}, has the structure of a Hausdorff log manifold.  Given a normal function $\nu : S \to \cJ$, that is admissible with respect to $\olS \supset S$, there exists a choice of $\Sigma'$ so that $\nu$ extends to $\nu_{\Sigma'} : \olS \to \cJ_{\Sigma'}$.
\end{theorem}

\begin{remark}
The space $\cJ_{\Sigma'}$ is a \emph{N\'eron model}.  In the geometric situation above, the N\'eron model encodes the degeneration of the intermediate Jacobians at points of $\partial S = \olS \bs S$.

Similar results were established by Nakayama \cite{MR3086749, MR3532114} subject to strong constraints on the nilpotent cones $\s$ and the monodromy group $\Gamma$.  Theorem \ref{T:neron} removes these constraints.
\end{remark}

\begin{proof}
The theorem is a corollary of Theorems \ref{T:main3} and \ref{T:main1-mixed}, and the work \cite{MR3084721}.

Let $\mathbf{U}$ be the kernel of the projection $\tAut(V_\bZ') \to \tAut(V_\bZ) \times \tAut(\bZ(0))$.  From $\Sigma$ and a choice of subgroup $\bfU(\bZ) \leq \Upsilon \leq \bfU(\bQ)$, subject to certain properties, we may construct a $\Gamma'$--strongly compatible weak fan $\Sigma' \subset \tEnd(V'_\bQ,W)$, \cf~\cite[\S4]{MR2721860}, so that $\pi: \Gamma' \bs \cM \to \Gamma \bs D$ extends to a map $\pi_\Sigma : \Gamma'\bs \cM_{\Sigma'} \to \Gamma \bs D_\Sigma$ of log manifolds.  Then we have the fibre product \eqref{E:fp}, \cf~\cite{MR2465224, MR3084721}, which may be viewed as a completion of \eqref{E:fp*}, and has the structure of a Hausdorff log manifold.

An admissible normal function $\nu : S \to \cJ$ is equivalent to an extension $0 \to \mathbb{V} \to \mathbb{V}' \to \bZ(0) \to 0$ in the category of admissible variations of mixed Hodge structure \cite{MR1374710}.  Let $\Sigma'$ be the fan of $\Sigma$ in Theorem \ref{T:main1-mixed}.  Then the extension $\nu_{\Sigma'}$ follows from Theorems \ref{T:main1} and \ref{T:main1-mixed}.
\end{proof}

\begin{remark}
When $\Upsilon = \mathbf{U}(\bZ)$, the space $\cJ_{\Sigma'}$ is the \emph{connected N\'eron model}.
\end{remark}
\section{Proof of Theorem \ref{T:BKTplus}} \label{S:prf-BKT+}

This is the main technical part of the paper in which we will prove Theorem \ref{T:BKTplus}. This section is organized as follows: We first review several known results from \cite{MR840721}, \cite{MR4155216} and \cite{MR3803710}, then provide a refinement of \cite[Theorem 1.5]{MR4155216} as well as a variant of \cite[Proposition 4.7]{MR3803710}. Necessary background on Siegel sets and generalizations can be found in Section \ref{S:appendix}.

\subsection{The $\tSL_2$-orbit theorem and Bakker--Klingler--Tsimerman's finiteness result}

In this section we review the classical $\tSL_2$-orbit theorem from \cite{MR840721} as well as \cite[Theorem 1.5]{MR4155216}. For the period domain $D$, recall $D\simeq G/M$ where $M\leq G$ is the stabilizer of some chosen base point $\varphi\in D$. 

Let \(X\) denote the symmetric space of positive definite symmetric bilinear
forms on \(V_{\mathbb R}\) with fixed determinant $1$. Thus $X\simeq SL(V_{\mathbb R})/K$ where \(K\) is the stabilizer of a chosen positive definite form. We may choose $M\leq K$, then there is a canonical $G$-equivariant morphism called the Hodge metric map:
\begin{equation}\label{E:Hodgemetric}
D=G/M\rightarrow \tSL(V_\bR)/K=:X, \ \varphi\mapsto Q(C_{\varphi}(-),-)
\end{equation}
where $C_{\varphi}\in \mathrm{Aut}(V_\bR)$ is the Weil operator associated to $\varphi\in D$. For any $\varphi\in D$ we denote $\varphi'\in X$ as its image. We now recall the following lemma which is a special case of \cite[Theorem 1.5]{MR4155216} when we assume the local period map agrees with some nilpotent orbit:

\begin{lemma}\label{L:BKTmain}
    The image of $\exp(\bi \mathfrak{H}_\s^+)\cdot F_\mathrm{id}\subset D$ in $X$ via the Hodge metric map is contained in the union of finitely many Siegel sets of $X$.
\end{lemma}

 Since by \cite[Lemma 7.5]{MR0147566} the preimage of any Siegel set of $X$ via the Hodge metric map is contained in the union of finitely many Siegel sets of $D$, Lemma \ref{L:BKTmain} implies:

\begin{theorem}\label{T:BKTmain}
   The subset $\exp(\bi \mathfrak{H}_\s^+)\cdot F_\mathrm{id}\subset D$ is contained in the union of finitely many Siegel sets in $D$.  
\end{theorem}


We need finer descriptions about the Siegel triples associated to Siegel sets in Lemma \ref{L:BKTmain} and Theorem \ref{T:BKTmain}. Useful references are \cite[Sec. 4]{MR4155216}, \cite[Chap. 1]{MR1046630} and \cite[7.5-7.7]{MR0147566}. See also \cite{Urb24}. We also need part of Cattani--Kaplan--Schmid's multivariable $\tSL_2$-orbit theorem \cite[(4.20)]{MR840721} which can be summarized as follows.
\begin{theorem}\label{T:CKSSL2}
Associated to the $\bR$-split MHS $(W=W(\sigma), F_\mathrm{id})$ with nilpotent cone $\s=\langle N_1,...,N_k\rangle$ and a chosen order $y_1\geq \cdots \geq y_k$ is a choice of commuting real $\tSL_2$ triples
\[
[\hat{N}_1,\hat{Y}_1,\hat{N}^+_1],\ldots,[\hat{N}_k,\hat{Y}_k,\hat{N}^+_k]
\]
such that for any $1\leq l\leq k$,
\begin{enumerate}
    \item $W(\sum_{1\leq j\leq l}N_j)=W(\sum_{1\leq j\leq l}\hat{N}_j)=:W_l$. In particular, all $W(\sum_{1\leq j\leq l}\hat{N}_j)$ are defined over $\bQ$.
    \item There exist $\hat{F}_1,...,\hat{F}_k\in \check{D}$ such that each $(W_j, \hat{F}_j)$ is an $\bR$-split MHS polarized by $\sum_{1\leq r\leq j}\hat{N}_r$ and graded by $\bfY_j:=\sum_{1\leq r\leq j}\hat{Y}_r$. 
    \item $\hat{F}_j=\exp(\bi (\hat{N}_{j+1}+...+\hat{N}_k))F_\mathrm{id}$. In particular, $(W_k,\hat{F}_k)=(W,F_\mathrm{id})$.
\end{enumerate}
\end{theorem}

The commuting semisimple grading elements $\bfY_1,...,\bfY_k$ and their adjoint operators give multigradings on $V_\bR$ and $\mathfrak{g}_\bR$ by:
\begin{eqnarray}\label{E:MultiGrading}
V_\bR&=&\bigoplus_{d_1,...,d_k \in \bZ}V^{d_1,...,d_k}, 
\\
\nonumber \mathfrak{g}_\bR&=&\bigoplus_{d_1,...,d_k \in \bZ}\mathfrak{g}^{d_1,...,d_k},
\end{eqnarray}
where each $d_j$ is the eigenvalue of $\bfY_j$ or $\mathrm{ad}_{\bfY_j}$. We also denote the weight filtrations on $V$ and $\fg$ induced by each $\bfY_j$ as $W_j$ and $W_j(\fg)$. 

Let $\mathfrak{p}_\cap:=\oplus_{d_i\leq 0}\mathfrak{g}^{d_1,...,d_k}=\bigcap_{1\leq j\leq k}W_j^{0}(\fg)$ and $P_\cap$ be the algebraic group with Lie algebra $\mathfrak{p}_\cap $.   It follows that $P_\cap$ is the $\bQ$-algebraic group stabilizing all weight filtrations $\{W_j\}_{1\leq j\leq k}$.  There are Levi-type decompositions:
\begin{equation}\label{E:SL2liealgebradecomp}
\mathfrak{g}_\cap =\mathfrak{g}^{0,...,0}\oplus(\oplus_{d_i\leq 0, \sum d_i<0}\mathfrak{g}^{d_1,...,d_k})=:{\mathfrak{l}_\cap}\oplus {\mathfrak{u}_\cap}
\end{equation}
in which ${\mathfrak{l}_\cap}=\mathfrak{g}^{0,...,0}=\bigcap_{1\leq j\leq r}\mathrm{ker}(\mathrm{ad}_{\hat{Y}_j})$ is reductive and ${\mathfrak{u}_\cap}$ is nilpotent, and its companion on the Lie group level:
\begin{equation}\label{E:SL2liegroupdecomp}
P_\cap=U_\cap\rtimes L_\cap, \ U_\cap=\exp(\mathfrak{u}_\cap), \ \mathrm{Lie}(L_\cap)=\mathfrak{l}_\cap.
\end{equation}
\begin{remark}
    Decompositions \eqref{E:SL2liealgebradecomp} and \eqref{E:SL2liegroupdecomp} are defined only over $\bR$ in general.
\end{remark}

Let $\mathcal{E}=\langle e_1,...,e_n\rangle$ be a basis of $V_\bR$ compatible with the multi-grading in the sense that any $e_j$ belongs to some $V^{d_1,...,d_k}$ in \eqref{E:MultiGrading}. Consider the set $\sO$ of linear relations on $\mathcal{E}$ (denoted as $\preceq$) satisfying the following condition: If $e\in V^{d_1,...,d_k}$ and $e'\in V^{d_1',...,d_k'}$ with $d_i
\leq d_i'$ for all $1\leq i\leq k$, then $e\preceq e'$. 

A choice of ordered basis $(\cE,\preceq)$ determines a canonical identification of $\tSL(V_\bR)$ with the matrix group $\tSL_N\bR$ for some $N\in \mathbb{N}$. Let $(P_\cE, A_\cE, K_\cE)_\preceq$ be the triple of subgroups of $\tSL(V_\bR)$ given by: $P_\cE$ is the upper-triangular subgroup, $A_\cE$ is the diagonal subgroup hence contains all $\{\bfY_j\}_{1\leq j\leq k}$, and $K_\cE$ is the standard orthogonal group. We also define $\mathfrak{H}_\sigma^{\geq}\subset \mathfrak{H}_\sigma^{+}$ as the sector:
\[
\mathfrak{H}_\sigma^{\geq}:=\{M=\sum_{1\leq j\leq k}y_jN_j\in  \mathfrak{H}_\sigma^{+}, \ y_1\geq y_2\geq...\geq y_k\geq 1\}.
\] 
Since the region $\mathfrak{H}_\sigma^{+}$ is covered by finitely many such sectors (up to choosing an element in $S_k$), it follows from the detailed proof of Theorem \ref{T:BKTmain} in \cite[Sec. 4]{MR4155216}, basic reduction theory introduced in \cite[Chap. 1]{MR1046630} and \cite[Lemma 7.5]{MR0147566} that Lemma \ref{L:BKTmain} can be now refined as follows:
\begin{theorem}\label{T:BKTrefined}
Fix an unordered basis $\cE$ of $V_\bR$ compatible with the multigrading \eqref{E:MultiGrading}, and $M\in \mathfrak{H}_\sigma^\geq$. Denote $\varphi:=\exp(\bi M)\cdot F_\mathrm{id}\in D$, then $(\exp(\bi\mathfrak{H}_\sigma^\geq)\cdot F_\mathrm{id})'\subset X$ is contained in finitely many Siegel sets of the form $\mathfrak{S}\cdot\varphi'\subset X$ in which $\fS =\Omega A_{\varphi', t'}K_{\varphi'}\subset \tSL(V_\bR)$, with $t' >0$, is a Siegel set with respect to the Siegel triple $(P,A_{\varphi'}, K_{\varphi'})$ with the following properties:
\begin{i_list_emph}
    \item 
    There exists an order $\preceq$ in $\sO$ and its corresponding Siegel triple $(P_\cE, A_\cE, K_\cE)_\preceq$ such that $P=P_\cE$ and $(A_{\varphi'},K_{\varphi'})$ is $P$-conjugate to $(A_\cE,K_\cE)$.
    \item 
    $K_{\varphi'}\leq \tSL(V_\bR)$ is the maximal compact subgroup stabilizing $\varphi'\in X$.
    \item 
    $K_\varphi=G\cap K_{\varphi'}$ and the Cartan involution of $K_{\varphi'}$ on $\tSL(V_\bR)$ restricts to the Cartan involution of $K_{\varphi}$ on $G$. 
    \item 
    There exists a Siegel triple $(P_1, A_1, K_{\varphi})$ of $G$ such that $A_1=A_{\varphi'}\cap G$ and $R_u(P_1)\leq R_u(P)\cap G$. 
\end{i_list_emph}
\end{theorem}
\begin{proof}
    By \cite[Sec. 4]{MR4155216}, $(\exp(\bi\mathfrak{H}_\sigma^\geq)\cdot F_\mathrm{id})'\subset X$ is covered by finitely many subsets of $X$ of the form
    \[
    \{h\in X, h \ \text{is} \  (\cE, \preceq,L)- \text{reduced for some } L>0\}.
    \]
    Therefore, (1) and (2) follow from the relationship between the sets of $(\cE, \preceq,L)$-reduced forms in $X$ and Siegel sets in $X$ introduced in \cite[Chap. 1]{MR1046630}, and (3) follows from \cite[Lemma 7.5]{MR0147566}.
\end{proof}
\begin{remark}
    To see the pair $G\leq\tSL(V_\bR)$ satisfies the assumption of \cite[Lemma 7.5]{MR0147566}, see \cite{MR2914941} and \cite{MR4614603}.
\end{remark}

Therefore, for any Siegel set $\fS\cdot \varphi'\subset X$ appearing in the finite collection of Theorem \ref{T:BKTrefined}, we may define the subset $E^+_\fS\subset X$ as:
\begin{equation}\label{E:coverbysiegel}
    E^+_\fS:=\fS\cdot\varphi'\cap (\exp(\bi\mathfrak{H}_\sigma^\geq)\cdot F_\mathrm{id})'.
\end{equation}
It follows from Theorem \ref{T:BKTrefined} that $(\exp(\bi\mathfrak{H}_\sigma^\geq)\cdot F_\mathrm{id})'$ and hence $(\exp(\bi\mathfrak{H}_\sigma^+)\cdot F_\mathrm{id})'$ are covered by finitely many sets of the form $E^+_\fS$.

\subsection{Siegel sets of reductive pairs}

In this section we review \cite[Theorem 4.1]{MR3803710} and its corrected version in \cite{MR4593766}. Moreover, by reviewing the details on the proof, we provide a more detailed version of this theorem which is critical for the proof of Theorem \ref{T:BKTplus}.

\begin{theorem}[{\cite{MR4593766}}]\label{T:OrrSch}
    Let $\bfH\leq \bfG$ be a pair of $\bQ$-reductive groups. If $(P_H,A_H,K_H)$ is a Siegel triple associated to some Siegel set $\fS_H\subset H$, and $K_G\leq G$ is a maximal compact subgroup satisfying:
    \begin{enumerate}
        \item $K_H\leq K_G$;
        \item The Cartan involution of $K_G$ on $G$ stabilizes $A_H$.
    \end{enumerate}
    Then there exists a Siegel set $\fS_G\subset G$ and a finite set $C\subset G$ such that $\fS_H\subset C\cdot \fS_G$. Moreover, the Siegel triple $(P_G, A_G, K_G)$ of $\fS_G$ satisfies $R_u(P_H)\leq R_u(P_G), A_H=A_G\cap H, K_H=K_G\cap H$.
\end{theorem}
\begin{remark}
    Though not needed in this paper, in \cite[Theorem 4.1]{MR3803710} the finite set $C$ can be chosen as a subset of $G(\bQ)$.
\end{remark}
To serve our purpose of applying this theorem, we need more information about the finite set $C\subset G$. We will review the necessary ingredients in the proof of \cite[Theorem 4.1]{MR3803710} below.

By \cite[Sec. 4A]{MR3803710} we may always reduce to the case that $A_H$ is split over $\bQ$. In this case, by \cite[Sec. 4B]{MR3803710} it is possible to choose a Siegel triple $(P_G, A_G, K_G)$ satisfying the conditions in Theorem \ref{T:OrrSch}.

Let $W^0:=N(A_G)/C(A_G)$ be the Weyl group associated with $A_G$ and $W\subset G$ be a set of chosen representatives\footnote{We work directly with the $\bR$-split $A_G$ here instead of choosing a $\bQ$-split conjugate of $A_G$ in \cite[Sec. 4D]{MR3803710}. This does not affect Theorem \ref{T:OrrSch}.}.   By \cite[Sec. 4F]{MR3803710}, there is a certain subset $C$ of $W$ determined only by $P_H$ and $P_G$ which is exactly the finite subset $C\subset G$ used in Theorem \ref{T:OrrSch}.

\begin{lemma}[{\cite[Proposition 4.7]{MR3803710}}]\label{L:OrrFiniteSet}
     Given $\fS_H=\Omega_HA_{H,t}K_H$ and a compatible Siegel triple $(P_G, A_G, K_G)$ of $G$ in the sense of Theorem \ref{T:OrrSch}. Let $Z=Z_G(A_H)$ and $P_Z=P_G\cap Z$, then $P_Z\leq Z$ is a parabolic subgroup. Let
    \[
    W^{\dagger}:=\{\omega\in W, \ \omega^{-1}R_u(P_H)\omega\leq R_u(P_G), \ \omega^{-1}R_u(P_Z)\omega\leq R_u(P_G)\},
    \]
    then there exists $s>0$ such that for any $\alpha\in A_{H,t}$ there exists $\lambda\in W^{\dagger}$ such that $\lambda^{-1}\alpha\lambda\in A_{G,s}$.  
\end{lemma}

\begin{remark}
    In \cite{MR3803710}, the finite set $C$ is chosen as a set of $\bQ$-representatives of the Weyl elements for some $\bQ$-split torus. However, if we do not insist on the rationality, we may directly use the finite set $W^\dagger$ in Lemma \ref{L:OrrFiniteSet}.
\end{remark}

More details on the proof of \cite[Theorem 4.1]{MR3803710} and its correction \cite[Theorem 1]{MR4593766} will be recalled in Section \ref{S:finalproof} when needed. In particular, we will need a variation of Lemma \ref{L:OrrFiniteSet} which is stated and proved in Section \ref{S:Orrplus}.

\subsection{Some technical lemmas}\label{S:Lemmas}
In this section we provide some technical results needed for the proof of Theorem \ref{T:BKTplus}. 

By Equation \eqref{E:UL=Z}, the groups $\bfZ_\s$ and $\bfP_W$ admit compatible Levi-type extensions as \eqref{E:cptbleLeviExt}: 
\begin{equation}\label{E:cptbleLeviExt2}
\begin{tikzcd}
1 \arrow[r] \arrow[d, equal] & \bfU_\s \arrow[r] \arrow[d, hookrightarrow] & \bfZ_\s \arrow[r] \arrow[d, hookrightarrow] & \bfL_\s \arrow[r] \arrow[d, hookrightarrow] & 1 \arrow[d, equal] \\1 \arrow[r] & \bfU_W \arrow[r] & \bfP_W \arrow[r] & \bfL_W \arrow[r] & 1
\end{tikzcd}
\end{equation}
The choice of LMHS $(W,F_\mathrm{id},\s)$ and its associated semisimple element $Y$ given by \eqref{E:Y-LMHS} gives the Levi-type decompositions $Z_\s=U_\s\rtimes L_{\s}$ and $P_W=U_W\rtimes L_{W}$ with $L_{\s}\leq L_{W}$ as splittings of \eqref{E:cptbleLeviExt2}. Note that since we assume $(W,F_\mathrm{id},\s)$ splits over $\bQ$ as Corollary \ref{C:rat}, $Y$ equals to $\bfY_k$ in Theorem \ref{T:CKSSL2}, and as a consequence, these decompositions are defined over $\bQ$. 

Suppose $\mathfrak{S}_\sigma\subset L_\sigma$ is a Siegel set associated to the Siegel triple $(P_\sigma, A_\sigma, K_\sigma)$ and of the form:
\begin{equation}\label{E:SiegelLsigma}
 \fS_\s:=\Omega_\s A_{\sigma,t} K_\sigma \subset L_\s
\end{equation}
for some $t>0$, where $K_\s\leq L_\s$ is defined in Lemma \ref{L:compatibleK}, $A_\sigma$ is conjugated to some maximal $\bQ$-split torus of $P_\sigma$ and stabilized by the Cartan involution of $K_\sigma\leq L_\sigma$. For any $M\in \s$, let $K_M\leq G$ be the maximal compact subgroup containing the stabilizer of $\exp(\bi M)\cdot F_\mathrm{id}\in D$. Lemma \ref{L:compatibleK} immediately implies:
\begin{corollary}\label{C:torusstab}
    For any $M\in \sigma$, $K_\s\leq K_M$ and the subgroups $A_{\s}, L_{\s}, \exp(\bR Y), L_{W}$ of $P_W$ are all stabilized by the Cartan involution of $K_M$ on $G$.
\end{corollary}
Passing from $G$ to $\tSL(V_\bR)$ and $D$ to $X$, let $P_W'\leq \tSL(V_\bR)$ be the parabolic determined by the weight filtration $W$ on $V_\bR$, and $L_W'\leq P_W'$ be the centralizer of $Y$. Since we assume $(W,F_0)$ is a $\bQ$-split LMHS graded by $Y$ as Corollary \ref{C:rat}, the group $L_W'\leq \tSL(V)$ is defined over $\bQ$. We also let $K_M'\leq \tSL(V_\bR)$ be the stabilizer of $(\exp(\bi M)\cdot F_\mathrm{id})'\in X$. We have the following corollary as the companion of Corollary \ref{C:torusstab}:

\begin{corollary}\label{C:torusstabSL}
    For any $M\in \sigma$, $K_\s\leq K_M'$ and the subgroups $A_{\s}, L_{\s}, \exp(\bR Y), L_{W}'$ of $P_W'$ are all stabilized by the Cartan involution of $K_M'$ on $\tSL(V_\bR)$.
\end{corollary}
The next important observation is:
\begin{lemma}\label{L:centralizer}
    We have $\mathfrak{z}_\sigma\subset \mathfrak{p}_\cap$ and $\mathfrak u_\sigma\subset \mathfrak{u}_\cap$. Moreover, the image of \(A_\sigma\) in the reductive quotient $P_\cap/U_\cap$ is a \(\bQ\)-split torus. 
\end{lemma}
\begin{proof}
    It follows directly from (1) of Theorem \ref{T:CKSSL2} and the fact that for any $1\leq l\leq k$, 
    $\fz(\sum_{1\leq j\leq l}N_j)\subset \oplus_{d_l\leq 0}\fg^{d_1,...,d_k}=W_l^{\leq 0}(\fg)$. The statement $\mathfrak u_\sigma\subset \mathfrak{u}_\cap$ follows from the fact $\mathfrak u_\s\subset \mathfrak{u}_W\cap \mathfrak{p}_\cap$.
\end{proof}
By the properties of Siegel sets (See Section \ref{S:appendix}), we may assume $\mathfrak{Z}_\sigma\subset Z_\s$ in Corollary \ref{C:fund-set} takes the form 
\begin{equation}\label{E:FundSetCent}
\mathfrak{Z}_\sigma=C_\s\cdot \wt{\Omega}_\sigma A_{\sigma,t} K_\sigma, \ \wt{\Omega}_\sigma=\Omega_\sigma'\cdot \Omega_\sigma 
\end{equation}

 where both $\Omega_\sigma'\subset U_\sigma$ and $\Omega_\sigma\subset P_\sigma$ are compact, and $A_{\sigma,t}$ is determined by \eqref{E:PosCone} for some fixed $t>0$, and $C_\s\subset Z_\s$ is a finite subset. In other words, 
 \[
 \mathfrak{Z}_\sigma=C_\s\cdot\Omega_\sigma'\cdot \fS_\s
 \]
 where $\fS_\s$ is defined as \eqref{E:SiegelLsigma}. We will show the following lemma:
\begin{lemma}\label{L:ChoosingSiegel}
    In Theorem \ref{T:BKTrefined} we may choose the ordered basis $(\mathcal{E}, \preceq)$ to satisfy the condition: The Siegel triple $(P, A_{\varphi'}, K_{\varphi'})$ associated to $\fS\subset \tSL(V_\bR)$ satisfies:
    \begin{enumerate}
    \item $U_\sigma\rtimes R_u(P_\sigma)\subset R_u(P)$,
    \item $A_\sigma= A_{\varphi'}\cap L_\s$,
    \item $K_\s=K_{\varphi'}\cap L_\s$ and the Cartan involution of $K_{\varphi'}$ stabilizes $A_\s$.
    \end{enumerate}
    In particular, the parabolic $P\leq \tSL(V_\bR)$ can be chosen to be defined over $\bQ$.  
\end{lemma}

\begin{proof}
    According the construction of the order $\preceq$, we have $P_\cap\leq P$ and $U_\cap\leq R_u(P)$. Lemma \ref{L:centralizer} and the choice of $(\cE, \preceq)$ in Theorem \ref{T:BKTrefined} immediately imply $\mathfrak{u}_\sigma\subset \mathfrak{u}_\cap\subset \mathrm{Lie}(R_u(P))$. Moreover, for any $(d_1,...,d_k)\in \bZ^k$, the space
    \[
    \mathrm{Gr}^{d_1,...,d_k}V:=\frac{\bigcap_{1\leq j\leq k}W_j^{\leq d_j}V}{\sum_{1\leq i\leq k}(W_i^{\leq d_i-1}V\cap\bigcap_{j\neq i}W_j^{\leq d_j}V)}
    \]
    is defined over $\bQ$, and  the image of $\mathfrak{a}_\sigma$ in $\mathfrak{p}_\cap/\mathfrak{u}_\cap$ acts on $\mathrm{Gr}^{d_1,...,d_k}V$ as a $\bQ$-split torus. It is possible to choose an ordered $\bQ$-basis $\mathcal{E}^{d_1,...,d_k}_0$ of $\mathrm{Gr}^{d_1,...,d_k}V$ diagonalizing this action, and the image of $R_u(P_\sigma)$ in $\mathfrak{p}_\cap/\mathfrak{u}_\cap$ acts through the upper-triangular group determined by $\mathcal{E}^{d_1,...,d_k}_0$.
    
    Since every $\mathcal{E}^{d_1,...,d_k}_0\subset \mathrm{Gr}^{d_1,...,d_k}V$ has a unique lift $\mathcal{E}^{d_1,...,d_k}\subset V^{d_1,...,d_k}$, we may take the union of all such $\mathcal{E}^{d_1,...,d_k}$ to obtain an ordered basis $\mathcal{E}$ subject to the relation $\preceq$. This $\mathcal{E}$ satisfies the required conditions in this lemma and Theorem \ref{T:BKTrefined} and hence establishes (1). (2) is clear from the construction of $\cE$ and (3) follows from Corollary \ref{C:torusstabSL}. By construction, $\mathcal{E}$ is $P$-conjugate to a basis in $V_\bQ$ as each $\mathcal{E}^{d_1,...,d_k}$ is $U_\cap$-conjugate to elements in $V_\bQ$, hence $P$ itself must be defined over $\bQ$.
\end{proof}
For any Siegel set $\fS=\Omega A_{\varphi',t'}K_{\varphi'}$ in $\tSL(V_\bR)$ with Siegel triple $(P,A_{\varphi'}, K_{\varphi'})$, and any $\psi'\in E^+_\fS$ as \eqref{E:coverbysiegel}, we may write $\psi'=u_\psi a_\psi\cdot \varphi'$ for some $u_\psi\in\Omega \subset R_u(P)$ and $a_\psi\in A_{\varphi',t'}$.
\begin{lemma}\label{L:UnipCent}
    We have $u_\psi\in L_W'$ and $u_\psi$ centralizes $A_\s$.
\end{lemma}

\begin{proof}
    By Lemma \ref{L:compatibleK}, for any $\psi'\in E^+_\fS$ the Cartan involution of $K_{\psi'}$ acts on $Y$ as multiplying by $-1$, hence $u_\psi$ must centralize $Y$ hence lie in $L_W'$. The unique maximal split torus in $P$ stabilized by the Cartan involution of $K_{\psi'}$ is $A_{\psi'}=u_\psi A_{\varphi'}u_\psi^{-1}$. By Lemma \ref{L:compatibleK} we have $A_\s\subset A_{\varphi'}\cap A_{\psi'}$ as $A_\s$ is stabilized by the Cartan involution of $K_{\psi'}$ for any $\psi'\in E^+_\fS$. Since $u_\psi\in R_u(P)$, $A_{\varphi'}\cap u_\psi A_{\varphi'}u_\psi^{-1}$ must lie in the centralizer of $u_\psi$.   Indeed, if \(a\in A_{\varphi'}\cap u_\psi A_{\varphi'}u_\psi^{-1}\), then
\(a=u_\psi b u_\psi^{-1}\) for some \(b\in A_{\varphi'}\). Comparing the
Levi and unipotent factors in \(P=R_u(P)\rtimes A_{\varphi'}\)  gives
\(a=b\) and \(a u_\psi a^{-1}=u_\psi\). The claim follows.
\end{proof}

\subsection{An alternative Orr-type theorem}\label{S:Orrplus}


Except for preparations in Section \ref{S:Lemmas}, we will need a modified version of Lemma \ref{L:OrrFiniteSet}. Under the settings of Theorem \ref{T:OrrSch} and Lemma \ref{L:OrrFiniteSet}, suppose for some $t,t'>0$, we have 
\[
\fS_H=\Omega_HA_{H,t}K_H, \ \fS_G=\Omega_GA_{G,t'}K_G
\]
with respect to Siegel triples $(P_H, A_H, K_H)$ and $(P_G, A_G, K_G)$.
\begin{theorem}\label{T:Orrplus}
    There exists $s>0$ depending only on $H\leq G, t,t'$ such that for any $a=a_Ha_G\in A_{H,t} A_{G,t'}$, there exists $\omega\in W^\dagger$ in Lemma \ref{L:OrrFiniteSet} such that $\omega^{-1} a\omega\in A_{G,s}$.
\end{theorem}

\begin{proof}
The main idea of the proof is similar to \cite[Prop. 4.7]{MR3803710}, but we also apply necessary modifications. Let $\Phi_G=\Phi(G,A_G), \Phi_G^+, \Delta_G$ be the sets of relative roots, positive roots, and simple roots determined by
\((P_G,A_G)\). Let
\[
\Phi_H^+:=\Phi(A_H,P_H)
\]
be the set of positive \(A_H\)-weights occurring in \(\Lie(U_H)\).

Define 
\[
\Theta
:=
\left\{
\lambda\in\Phi_G^+:
\lambda|_{A_H}\in \Phi_H^+\cup\{0\}
\right\}.
\]
Since \(\Theta\subset \Phi_G^+\), the set \(\Theta\) is
\(\bR_{>0}\)-independent; In particular, this means for any subset $\{\lambda_1,...,\lambda_n\}\subset \Theta$, if there exist non-negative numbers $
\alpha_1,...,\alpha_n$ with $\sum_i \alpha_i\chi_i=0$, then we must have all $\alpha_i=0$. 

We first prove a uniform lower bound for the roots in \(\Theta\) on
\(A_{H,t} A_{G,t'}\).     Let
\[
a=a_Ha_G\in A_{H,t} A_{G,t'}
\]
and let \(\lambda\in\Theta\). If
\[
\lambda|_{A_H}\in\Phi_H^+,
\]
then \(\lambda|_{A_H}\) is a positive \(H\)-root. Since
\(a_H\in A_{H,t}\), there is a constant \(\eta_H(t)>0\), depending only
on \(t\) and the root data of \(H\), such that
\[
\lambda(a_H)=(\lambda|_{A_H})(a_H)\ge \eta_H(t).
\]
Also, since \(\lambda\in\Phi_G^+\) and \(a_G\in A_{G,t'}\),
there is a constant \(\eta_G(t')>0\), depending only on \(t'\) and the
root data of \(G\), such that
\[
\lambda(a_G)\ge \eta_G(t').
\]
Therefore
\[
\lambda(a)=\lambda(a_Ha_G)
=
\lambda(a_H)\lambda(a_G)
\ge
\eta_H(t)\eta_G(t').
\]
If instead $\lambda|_{A_H}=0$, then $\lambda(a_H)=1$ and again
\[
\lambda(a_G)\ge \eta_G(t').
\]
Thus in all cases there exists a number $\eta>0$ depending only on \(H\leq G\) and on \(t,t'\), such that $\lambda(a)\ge \eta$
for all $a\in A_{H,t} A_{G,t'}$ and all $\lambda\in\Theta$.

Put $\eta_0:=\min(1,\eta)$. For any \(a\in A_{H,t} A_{G,t'}\), choose a subset $\Psi_a\subset \Phi_G$ maximal subject to the following two conditions: $\Theta\cup\Psi_a$
is $\bR_{>0}$-independent, and $\psi(a)\ge \eta_0$ for every $\psi\in\Psi_a$.  Such a maximal set exists because \(\Theta\) is already $\bR_{>0}$-independent. 

Choose an additive order \(>_a\) on
\(X^*(A_G)_\mathbf R\) such that every element of $\Theta\cup\Psi_a
$ is positive with respect to this order. Define
\[
\Phi_a^+
:=
\{\chi\in\Phi_G:\chi>_a0\},
\]
then \(\Phi_a^+\) is a positive system of \(\Phi_G\). Hence there exists $\omega = \omega_a \in W_G$ such that $\Phi_a^+=\omega\Phi_G^+$. Since \(\Theta\subset\Phi_a^+\), we have $\omega^{-1}\Theta\subset\Phi_G^+$. 

We claim that every \(\chi\in\Phi_a^+\) is a non-negative real combination of
elements of \(\Theta\cup\Psi_a\). If \(\chi\in\Theta\cup\Psi_a\), there is
nothing to prove. Suppose $\chi\notin\Theta\cup\Psi_a$. Since all elements of
\[
\Theta\cup\Psi_a\cup\{\chi\}
\]
are positive for the same order \(>_a\), this enlarged set is still
\(\bR_{>0}\)-independent. If $\chi(a)\ge \eta_0$, then \(\chi\) could be added to \(\Psi_a\), contradicting maximality.
Thus $\chi(a)<\eta_0$.

Since \(0<\eta_0\le 1\), we get
\[
(-\chi)(a)=\chi(a)^{-1}>\eta_0^{-1}\ge \eta_0.
\]
Hence \(-\chi\) satisfies the same value condition. By maximality of
\(\Psi_a\), the set $\Theta\cup\Psi_a\cup\{-\chi\}$ cannot be \(\bR_{>0}\)-independent. Therefore there is a linear relation
\[
c(-\chi)+\sum_i c_i\tau_i=0, \ c,c_i>0, \ \tau_i\in\Theta\cup\Psi_a
\]
which implies $\chi=\sum_i \frac{c_i}{c}\tau_i$. This proves the claim.

Since the root system \(\Phi_G\) is finite, there exists a constant
\(M>0\) depending only on \(\Phi_G\), such that every
\(\chi\in\Phi_a^+\) can be written as
\[
\chi=\sum_i b_i\tau_i,
\
b_i\ge0,\ \tau_i\in\Theta\cup\Psi_a,
\]
with $\sum_i b_i\le M$. For every \(\tau_i\in\Theta\cup\Psi_a\), we have $\tau_i(a)\ge \eta_0$. Therefore
\[
\chi(a)
=
\prod_i \tau_i(a)^{b_i}
\ge
\eta_0^{\sum_i b_i}
\ge
\eta_0^M.
\]
Set $s:=\eta_0^M>0$, then $\chi(a)\ge s$ for every $\chi\in\Phi_a^+$. Let \(\alpha\in\Delta_G\). Since $\Phi_a^+=\omega\Phi_G^+$, we have $\omega\alpha\in\Phi_a^+$. Therefore
\[
\alpha(\omega^{-1}a\omega)
=
(\omega\alpha)(a)
\ge s.
\]
Since this holds for every simple root \(\alpha\in\Delta_G\), we obtain $\omega^{-1}a\omega\in A_{G,s}$.  

It remains to check that \(\omega\in W^\dagger\). Since $\omega^{-1}\Theta\subset\Phi_G^+$, every root in \(\Theta\) is sent by \(\omega^{-1}\) to a positive
root. We consider \(U_Z\) first. The roots of \(U_Z\) are exactly the positive
ambient roots \(\lambda\in\Phi_G^+\) satisfying $\lambda|_{A_H}=0$ which lie in \(\Theta\), hence
\[
\omega^{-1}U_Z\omega\leq U_G.
\]
Next consider \(U_H\). Let \(X\in\Lie(U_H)\), and decompose \(X\) into
\(A_G\)-root components:
\[
X=\sum_\lambda X_\lambda,
\
X_\lambda\in\mathfrak g_\lambda.
\]
If \(X_\lambda\ne0\), then its \(A_H\)-weight is \(\lambda|_{A_H}\).
Because \(X\in\Lie(U_H)\), this restriction lies in \(\Phi_H^+\), thus $\lambda\in\Theta$. Since \(\omega^{-1}\Theta\subset\Phi_G^+\), all the $G$--root components of
\(\operatorname{Ad}(\omega^{-1})X\) lie in positive \(G\)-root spaces.
Thus
\[
\operatorname{Ad}(\omega^{-1})\Lie(U_H)\subset \Lie(U_G),
\]
or equivalently $\omega^{-1}U_H\omega\leq U_G$ and therefore $\omega\in W^\dagger$. The proof is complete.  
\end{proof}

\subsection{Proof of Theorem \ref{T:BKTplus}}\label{S:finalproof}

We are ready to prove Theorem \ref{T:BKTplus}. Fix a Siegel set $\fS_\s=\Omega_\s A_{\s,t} K_\s\subset L_\s$ with Siegel triple $(P_\s, A_\s, K_\s)$ as Equation \eqref{E:SiegelLsigma}, an ordered basis $(\cE, \preceq)$ of $\mathrm{SL}(V_\bR)$, and a Siegel set $\fS\subset \mathrm{SL}(V_\bR)$ with Siegel triple $(P, A_{\varphi'}, K_{\varphi'})$ corresponding to the ordered basis $(\cE, \preceq)$ satisfying all conditions in Theorem \ref{T:BKTrefined} and Lemma \ref{L:ChoosingSiegel}.

Since $\fS_\s\subset P_W'$ and $E^+_\fS\subset P_W'\cdot \varphi'$ by Lemma \ref{L:UnipCent}, we have 
\[
\fS_\s\cdot E^+_\fS\subset P_W'\cdot\varphi'\subset X.
\] 
To prove Theorem \ref{T:BKTplus}, it is enough to show $\fS_\s\cdot E^+_\fS$ is contained in finitely many generalized Siegel sets in $P_W'\cdot\varphi'$ in the sense of Lemma \ref{L:GeneralizedsiegelpropertiesD}. Indeed, it is enough to show the following theorem:
\begin{theorem}\label{T:BKTplus2}
     There exists a finite subset $C\subset \mathrm{SL}(V_\bR)$ and a Siegel set $\wh{\fS}=\wh{\Omega} A_{\varphi', s}K_{\varphi'}$ with respect to the Siegel triple $(P,A_{\varphi'}, K_{\varphi'})$ for some $s>0$, such that for any $\psi'\in E^+_\fS$ we have
    \[
    \fS_\s\cdot \psi'\subset C\cdot\wh{\fS}\cdot\varphi'.
    \]
    In particular, we may choose $C\subset K_{\varphi'}$ and normalizes $A_{\varphi'}$.
\end{theorem}

\subsubsection{Proof of Theorem \ref{T:BKTplus} assuming Theorem \ref{T:BKTplus2}}

For any $\psi'\in E^+_\fS$ let $K_{W,\psi'}'\leq L_W'$ be the stabilizer of $\psi'$. Note the fact that $A_{\varphi'}$ is a maximal $\bR$-split torus in either $\tSL(V_\bR)$ or $L_W'$,  it follows from Corollary \ref{C:torusstabSL} and Theorem \ref{T:BKTrefined} that we have:
\begin{corollary}\label{C:SLtoLW}
   For the Siegel triple $(P,A_{\varphi'}, K_{\varphi'})$ of $\tSL(V)$, the triple $(P_{L_W'}:=P\cap L_W', A_{\varphi'}, K_{W,\varphi'}:=K_{\varphi'}\cap L_W')$ is a Siegel triple of $L_W'$ compatible with $(P_\s, A_\s, K_\s)$ in the sense of Theorem \ref{T:BKTrefined}. 
\end{corollary}
Corollary \ref{C:SLtoLW} together with \cite[Lemma 7.5]{MR0147566} imply:
\begin{lemma}\label{L:SiegelContain}
    For any Siegel set $\fS\subset\tSL(V_\bR)$ with respect to the Siegel triple $(P,A_{\varphi'}, K_{\varphi'})$, there exist finitely many Siegel sets $\{\fS_{W,i}\}$ of $L_W'$ with respect to the Siegel triple $(P_{L_W'}, A_{\varphi'}, K_{W,\varphi'})$ such that
    \[
    \fS\cap L_W'\subset\bigcup_i \fS_{W,i}
    \]
    and
    \[
    (\fS\cdot \varphi')\cap (L_W'\cdot\varphi')\subset\bigcup_i (\fS_{W,i}\cdot\varphi').
    \]
\end{lemma}
By Lemma \ref{L:UnipCent} and \ref{L:SiegelContain}, Theorem \ref{T:BKTplus2} implies there exists a finite collection of Siegel sets $\{\fS_{W,j}\}$ of $L_W'$ with respect to the Siegel triple $(P_{L_W',j}, A_{\varphi'}, K_{W,\varphi'}$) such that for any $\psi'\in E^+_\fS$,
\begin{eqnarray*}
\fS_\s\cdot\psi'&=&\fS_\s\cdot u_\psi a_\psi\cdot\varphi'\\
&\subset& C\cdot \wh{\fS}\cdot\varphi'\cap L_W'\cdot\varphi'\\
&\subset& \bigcup_j (\fS_{W,j}\cdot \varphi'),
\end{eqnarray*}
where the last containment follows from the statement that $C\subset K_{\varphi'}\cap N_{
\tSL(V_\bR)}(A_{\varphi'})$ which implies $c\cdot \wh{\fS}$ is a Siegel set with respect to the Siegel triple $(cPc^{-1},A_{\varphi'}, K_{\varphi'})$ for any $c\in C$. 

It follows that from Equation \eqref{E:FundSetCent}, there is a finite subset $C_\s\subset L_\s$ such that
\begin{eqnarray}\label{E:Siegelcontain2}
\bigcup_{\psi'\in E^+_\fS}\mathfrak{Z}_\s\cdot \psi'&=& \bigcup_{\psi'\in E^+_\fS}C_\s\cdot \Omega_\s'\cdot \fS_\s\cdot\psi'\\
\nonumber &\subset&C_\s\cdot\bigcup_j(\Omega_\s'\cdot\fS_{W,j}\cdot\varphi').
\end{eqnarray}
According to Lemma \ref{L:GeneralizedsiegelpropertiesQ} and Lemma \ref{L:GeneralizedsiegelpropertiesD}, the RHS of Equation \eqref{E:Siegelcontain2} is independent of the choice of $\psi'\in E^+_\fS$ and is contained in a finite union of generalized Siegel sets in $P_W'\cdot\varphi'$. By taking the preimage under the Hodge metric map \eqref{E:Hodgemetric}, the same arguments as \cite[Sec. 4.5]{MR4155216} and \cite[Lemma 7.5]{MR0147566} imply
\[
\bigcup_{\psi'\in E^+_\fS}\mathfrak{Z}_\s\cdot \psi\subset P_W\cdot \varphi=U_W\cdot (L_W\cdot \varphi)
\]
is contained in finitely many generalized Siegel sets of $P_W\cdot \varphi$. For any Siegel set $\fS$ appearing in Theorem \ref{T:BKTrefined}, in Theorem \ref{T:BKTplus} we may take each $R_k$ and $B_{\s,k}$ to be the form 
\[
(\exp(\bi R_k)\cdot F_\mathrm{id})'=E^+_\fS, \  \mathfrak{B}_{\s,k}=\mathfrak{Z}_\s\cdot F_\mathrm{id} 
\]
where $\fS_\s$ takes the form as Equation \eqref{E:FundSetCent}, then 
\[
\exp(\bi R_k)\cdot\mathfrak{B}_{\s,k}=\bigcup_{\psi'\in E^+_\fS} \mathfrak{Z}_\s\cdot \psi\subset P_W\cdot\varphi 
\]
is contained in finitely many generalized Siegel sets in $P_W\cdot\varphi$. Theorem \ref{T:BKTplus} follows as there are only finitely many such Siegel sets $\fS$ by Theorem \ref{T:BKTmain}.

\subsubsection{Proof of Theorem \ref{T:BKTplus2}}

We first show the following lemma which will be proved at the end of this section:
\begin{lemma}\label{L:BKTplusLemma}
    There exists a Siegel set $\wh{\fS}=\wh{\Omega} A_{\varphi', s}K_{\varphi'}\subset \mathrm{SL}(V_\bR)$ with the same Siegel triple $(P,A_{\varphi'}, K_{\varphi'})$ as $\fS$ in Theorem \ref{T:BKTrefined} such that 
    \begin{equation}\label{E:Siegelcontainment}
    \Omega_\s A_{\s,t}\subset C\cdot\wh{\Omega}\cdot A_{\varphi', s}\cdot a_\psi^{-1}u_\psi^{-1}K_{\psi'}
    \end{equation}
    for any $\psi'=u_\psi a_\psi \cdot\varphi'\in E^+_\fS$.
\end{lemma}

Let $\fC\leq \mathrm{SL}(V_\bR)$ be the subgroup centralizing $A_\s$. By Corollary \ref{C:rat}, $A_\s$ is $P_\s$-conjugated to a maximal $\bQ$-split torus in $\bfL_\s$. It follows that $\fC\leq \tSL(V_\bR)$ is an $\bR$-reductive subgroup $P_\s$-conjugated to (the $\bR$-points of) some $\bQ$-reductive subgroup $\fC_\bQ$ of $\tSL(V)$. Following the constructions in \cite[Sec. 4B]{MR3803710}, we have:
 \begin{lemma}\label{L:ExtendedCentralizer}
      We have $A_{\psi'}\leq \fC$ and $P_\fC:=P\cap \fC$, $K_{\fC,\psi'}:=K_{\psi'}\cap\fC$ give a Siegel triple $(P_\fC, A_{\psi'}, K_{\fC,\psi'})$ of $\fC$\footnote{We may define Siegel triples of the $\bR$-reductive $\fC$ using the conjugates of Siegel triples in $\fC_\bQ$.}. Moreover, suppose $P_\s=R_u(P_\s)\cdot A_\s\cdot M_\s$ is the Langlands decomposition of $P_\s$, then $M_\s\leq \fC$ is a subgroup stabilized by the Cartan involution of $K_{\fC,\psi'}$ on $\fC$ and $K_{\psi'}$ on $\mathrm{SL}(V_\bR)$.
 \end{lemma}
 As a consequence, we may choose bounded $\Omega_{\s,u}\subset R_u(P_\s), \ \Omega_{\s,A}\subset A_\s, \ \Omega_{\s,M}\subset M_\s$ such that\[\Omega_\s\subset\Omega_{\s,u}\cdot \Omega_{\s,A}\cdot \Omega_{\s,M}.\] 
 Note that by Theorem \ref{T:OrrSch}, $\Omega_{\s,u}\cdot \Omega_{\s,A}\subset P$, but $\Omega_{\s,M}$ (hence $\Omega_\s$) need not be a subset of $P$ in general. However, since $\Omega_{\s,A}\cdot \Omega_{\s,M}\subset \fC$, we have the following lemma:

 \begin{lemma}\label{L:UniformCptCover}
     There exists a bounded $\Omega_{P,\fC}\subset P_\fC$ such that $\Omega_{\s,A}\cdot \Omega_{\s,M}\subset \Omega_{P,\fC}\cdot K_{\fC,\psi'}$ for any $\psi'=u_\psi a_\psi\cdot \varphi'\in E^+_\fS$.
 \end{lemma}

\begin{proof}
 Since $A_{\psi'}\leq \mathrm{SL}(V_\bR)$ is a maximal $\bR$-split torus stabilized by the Cartan involution of $K_{\fC,\psi'}$ on $\fC$ and $K_{\psi'}$ on $\mathrm{SL}(V_\bR)$, by Lemma \ref{L:ExtendedCentralizer} we have $A_{M_\s}:=A_{\psi'}\cap M_\s$ is a maximal $\bR$-split torus of $M_\s$ and there is a minimal $\bR$-parabolic subgroup $P_{M_\s}$ of $M_\s$ whose split torus is $A_{M_\s}$  with $R_u(P_{M_\s})\leq R_u(P)$. Moreover, note that by Lemma \ref{L:compatibleK}, $K_{M_\s}:=K_{\fC,\psi'}\cap M_\s$ is independent of the choice of $\psi'\in E^+_\fS$ (Indeed, $K_{M_\s}=K_\s\cap M_\s$).
 
 Therefore, $M_\s$ admits the Iwasawa decomposition $M_\s=R_u(P_{M_\s})\cdot A_{M_\s}\cdot K_{M_\s}$ and we may choose a compact subset $\Omega_{P,\fC}\subset P_\fC$ such that $\Omega_{P,\fC}$ contains $\Omega_{\s,A}$ and $(\Omega_{P,\fC}\cap M_\s)\cdot K_{M_\s}$ contains $\Omega_{\s,M}$. Note that the choice of $\Omega_{P,\fC}$ does not depend on the choice of $\psi'\in E^+_\fS$. The lemma follows.
 \end{proof}

 Next let $W$ be a fixed set of Weyl representatives for $A_{\varphi'}\leq P$. By \cite[Lemma 4.10]{MR3803710} we may assume the representatives in $W$ lie in $K_{\varphi'}$\footnote{We do not use multiple representatives $\omega_\bQ', \omega_\bQ, \omega_K$ as \cite[Sec. 4D-4F]{MR3803710} but just $\omega_K$. By \cite[Lemma 4.11]{MR3803710}, this change does not affect our arguments for $\omega\in W^\dagger$.}. For the compatible Siegel triples $(P_\s, A_\s, K_\s)$ and $(P, A_{\varphi'}, K_{\varphi'})$ associated to the reductive pair $L_\s\leq \tSL(V)$, let $W^\dagger\subset W$ be the subset as Lemma \ref{L:OrrFiniteSet}. For any $\psi'=u_\psi a_\psi\cdot \varphi'$, let $\omega_{\psi'}:=u_\psi a_\psi\omega a_\psi^{-1}u_\psi^{-1}\in K_{\psi'}$ and $W_{\psi'}:=u_\psi a_\psi Wa_\psi^{-1}u_\psi^{-1}\subset K_{\psi'}$.


 Let $\Omega_{P,\fC}\subset \fC$ be the bounded subset as in Lemma \ref{L:UniformCptCover}. From the Langlands decomposition $P_\fC=R_u(P_\fC)\cdot A_{\psi'}\cdot M_{\psi'}$, let 
 \[
 \Omega_{u,\fC}(\psi')\subset R_u(P_\fC), \ \Omega_{A,\fC}(\psi')\subset A_{\psi'}, \ \Omega_{M,\fC}(\psi')\subset M_{\psi'}
 \]
 be the projection of $\Omega_{P,\fC}$ on each factor, and let
\begin{equation}\label{E:LangDecomp}
 \widehat{\Omega_{P,\fC}}(\psi'):=\Omega_{u,\fC}(\psi')\cdot \Omega_{A,\fC}(\psi')\cdot \Omega_{M,\fC}(\psi')\subset P_\fC.
\end{equation}
Since by Lemma \ref{L:UnipCent}, $A_{\psi'}\cdot M_{\psi'}=\mathrm{Ad}_{u_\psi}(A_{\varphi'}\cdot M_{\varphi'})$
 for some $u_\psi$ lying in some bounded subset of $R_u(P_\fC)$, we may take some $\widehat{\Omega_{P,\fC}}\subset P_\fC$ containing $\widehat{\Omega_{P,\fC}}(\psi')$ for all $\psi'\in E^+_\fS$. Combining Lemma \ref{L:UniformCptCover} and Equation \eqref{E:LangDecomp} we have:
\begin{equation}\label{E:CoverLeviPart}
\Omega_{\s,A}\cdot \Omega_{\s,M}\subset \Omega_{P,\fC}\cdot K_{\fC,\psi'}\subset \widehat{\Omega_{P,\fC}}(\psi')\cdot K_{\fC,\psi'}\subset \widehat{\Omega_{P,\fC}}\cdot K_{\fC,\psi'}.
\end{equation}
 Next for any $\omega\in W^\dagger$ we define:
 \begin{equation}\label{E:CPTFactor}
\Omega_P(\psi',\omega):=\omega^{-1}\Omega_{\s,u}\cdot\widehat{\Omega_{P,\fC}}\cdot\omega_{\psi'}.
 \end{equation}
Notice that $\omega^{-1}\omega_{\psi'}=\omega^{-1}u_\psi a_\psi \omega a_\psi^{-1}u_\psi^{-1}$ and the facts that $u_\psi\in \fC$ and $\omega\in W^\dagger$, by the definition of $W^\dagger$ in Lemma \ref{L:OrrFiniteSet} we know $\Omega_P(\psi',\omega)$ is a bounded subset of $P$ (but depends on the choice of $\psi'\in E^+_\fS$). Moreover, by Equations \eqref{E:CoverLeviPart} and \eqref{E:CPTFactor} we have for any $\omega\in W^\dagger$,
 \begin{eqnarray}\label{E:HtoG}
     \nonumber \omega^{-1}\Omega_\s \cdot A_{\s,t}&\subset& \Omega_P(\psi',\omega)\cdot\omega_{\psi'}^{-1}\cdot A_{\s,t} \cdot K_{\fC,\psi'}\\
     &=& \Omega_P(\psi',\omega)(\omega_{\psi'}^{-1}\cdot A_{\s,t}\cdot\omega_{\psi'})(\omega_{\psi'}^{-1}\cdot K_{\fC,\psi'})\\
     \nonumber&\subset&\Omega_P(\psi',\omega)(\omega_{\psi'}^{-1}\cdot A_{\s,t}\cdot\omega_{\psi'})\cdot K_{\psi'}.
\end{eqnarray}

\begin{lemma}\label{L:ChoosingSiegel2}
    It is possible to choose compact $\wh{\Omega}\subset P$ and $A_{\varphi',s}$ for some $s>0$ such that for any $\psi'=u_\psi a_\psi\cdot \varphi'\in E^+_\fS$ and any $\mu\in \Omega_\s, a_\s\in A_{\s,t}$, there exists some $\lambda\in W^\dagger$ such that
    \[
    \lambda^{-1}\mu a_\s\in \wh{\Omega}\cdot A_{\varphi',s}\cdot a_\psi^{-1}u_\psi^{-1} K_{\psi'}.
    \]
\end{lemma}
\begin{proof}
    Apply Theorem \ref{T:Orrplus} to the pair $\bfL_\s\leq \tSL(V)$ and the corresponding Siegel sets $\fS_\s\subset L_\s, \fS\subset \tSL(V_\bR)$ we know there exists some $s>0$ such that for any $\mu a_\s\in \Omega_\s\cdot A_{\s,t}$ and $\psi'\in E^+_\fS$ with $a_\psi\in A_{\varphi',t'}$, there exists some $\lambda\in W^\dagger$ such that 
    \begin{equation}\label{E:OrrElement}
        \lambda^{-1}a_\s a_\psi\lambda\in A_{\varphi',s}.
    \end{equation}
    We also denote $\lambda_{\psi'}$ as the corresponding $\lambda$-conjugate in $W_{\psi'}^\dagger$. 
    
    Note that by Lemma \ref{L:UnipCent}, both \(u_\psi\) and \(a_\psi\) commute with \(a_\sigma\). By Equations \eqref{E:CPTFactor} and \eqref{E:HtoG} we have:
    \begin{eqnarray}\label{E:factorization}
        \lambda^{-1}\mu a_\s 
        &\in&\Omega_P(\psi',\lambda)\cdot (\lambda_{\psi'}^{-1}\cdot a_\s\cdot\lambda_{\psi'})\cdot K_{\psi'}\\
        \nonumber&=&(\lambda^{-1}\cdot\Omega_{\s,u}\cdot\widehat{\Omega_{P,\fC}}) \cdot a_\s\cdot\lambda_{\psi'}\cdot K_{\psi'}\\
        \nonumber&=&(\lambda^{-1}\cdot\Omega_{\s,u}\cdot\widehat{\Omega_{P,\fC}}\cdot \lambda)\cdot (\lambda^{-1}\cdot a_\s\cdot\lambda_{\psi'})\cdot K_{\psi'}\\
        \nonumber&=&\Omega_P(\varphi',\lambda)\cdot(\lambda^{-1} \cdot a_\s\cdot(u_\psi a_\psi \lambda a_\psi^{-1}u_\psi^{-1}))\cdot K_{\psi'}\\
        \nonumber&=&\Omega_P(\varphi',\lambda)\cdot(\lambda^{-1}u_\psi\lambda)\cdot(\lambda^{-1} a_\s a_\psi \lambda) \cdot a_\psi^{-1}u_\psi^{-1}K_{\psi'}.
    \end{eqnarray}
By Lemma \ref{L:OrrFiniteSet} and \ref{L:UnipCent}, since $\lambda\in W^\dagger$ and $u_\psi\in R_u(P_\mathfrak{C})$, we know the set 
\[\Omega_P(\varphi',\lambda)\cdot (\lambda^{-1}u_{\psi}\lambda)\] 
is a bounded subset of $P$. Since \(u_\psi\) is the unipotent coordinate of a point of the fixed Siegel
set \(\fS=\Omega A_{\varphi',t'}K_{\varphi'}\), the elements \(u_\psi\) range in a bounded subset of \(R_u(P_\fC)\). It is therefore possible to choose some bounded $\wh\Omega\subset P$ which does not depend on the choice of $\psi'\in E^+_\fS$, such that 
\begin{equation}\label{E:Boundfirstfactor}
    \Omega_P(\varphi',\lambda)\cdot(\lambda^{-1}u_\psi\lambda)\subset \bigcup_{\omega\in W^\dagger}\Omega_P(\varphi',\omega)\cdot (\omega^{-1}u_{\psi}\omega)\subset \wh\Omega.
\end{equation}
 Combining Equations \eqref{E:OrrElement} and \eqref{E:Boundfirstfactor} with Equation \eqref{E:factorization} proves the lemma.
\end{proof}

As a consequence of Lemma \ref{L:ChoosingSiegel2}, Lemma \ref{L:BKTplusLemma} immediately follows as we may choose $C=W^\dagger\subset K_{\varphi'}\cap N_{\tSL(V_\bR)}(A_{\varphi'})$. Theorem \ref{T:BKTplus2} and hence Theorem \ref{T:BKTplus} follow from Lemma \ref{L:BKTplusLemma} immediately by multiplying both sides of Equation \eqref{E:Siegelcontainment} on $\psi'$ with the fact $K_\s\leq K_{\psi'}$ for any $\psi'\in E^+_\fS$ in mind. The proofs of Theorem \ref{T:BKTplus2} and hence Theorem \ref{T:BKTplus} are now concluded. \hfill \qed

\section{Appendix: Siegel sets and fundamental sets} \label{S:appendix}


\subsection{Siegel sets for reductive groups}

We briefly review the theory of $\bQ$-reductive groups and Siegel sets. We will use notations from \cite{MR3803710}.

Let $\bfG$ be a reductive algebraic group defined over $\bQ$ and $G=\bfG(\bR)$. Let $\bfP\leq \bfG$ be a minimal $\bQ$-parabolic subgroup of $\bfG$ and let $P=\bfP(\bR)$. Let $A\leq P$ be an $\bR$-split torus which is $P$-conjugate to a maximal $\bQ$-split torus in $P$. Let $K\leq G$ be a maximal compact subgroup such that \textit{its associated Cartan involution $\theta_K$ stabilizes $A$}.

Notice that in this case, $A$ descends to the $\bQ$-split center in the Levi quotient $P/R_u(P)$ where $R_u(P)\leq P$  is the unipotent radical, and $L:=P\cap \theta_K(P)$ gives the unique Levi subgroup which is the centralizer of $A$ in $P$. 

Let $\Phi(A,G)$ be the roots of $A$ in $\mathrm{Lie}(G)$ and $\Phi^+(A,G)\subset \Phi(A,G)$ as the subset whose root spaces are contained in $\mathrm{Lie}(P)$. We may choose a set of simple roots $\Delta\subset \Phi^+(A,G)$ such that every root in $\Phi^+(A,G)$ is a linear combination of roots in $\Delta$ with non-negative coefficients.

\begin{definition}
    For a triple $(P, A, K)$ satisfying the conditions above, a Siegel set $\mathfrak{S}\subset G$ with respect to the Siegel triple $(P, A, K)$ is a set of the form:
    \begin{equation}
        \mathfrak{S}:=\Omega A_t K\subset G,
    \end{equation}
    where $\Omega\subset P$ is a compact subset, and for the given $t>0$, 
    \begin{equation}\label{E:PosCone}
        A_t:=\{a\in A^+ \ | \ \chi(a)\geq t, \ \forall \chi\in \Delta\}
    \end{equation}
    where $A^+\subset A$ is the connected component containing the identity.
\end{definition}

Suppose $D$ is a homogeneous variety which admitting a transitive $G$-action, and for any $F\in D$, the stabilizer $\mathrm{Stab}_F\leq G$ is compact. In particular, this is the case for the period domain $D$ parametrizing $\bZ$-PHS of a given type. 

\begin{definition}
    Suppose $\mathfrak{S}\subset G$ is a Siegel set with Siegel triple $(P,A,K)$. For a point $F\in D\cong G/\mathrm{Stab}_F$ such that $\mathrm{Stab}_F\leq K$, we call the set $\mathfrak{S}\cdot F\subset D$ a Siegel set of $D$.  
\end{definition}

The following Lemmas are classical results in reduction theory. See for example \cite{MR0244260} or \cite{MR2189882}.

\begin{lemma}\label{L:SiegelProperties}
    Let $\mathfrak{S}, \mathfrak{T}$ be Siegel sets in $G$, and $\Gamma$ be an arithmetic subgroup of $\mathrm{Aut}(D)$.
    \begin{enumerate}
        \item The set $\{\gamma\in \Gamma \ | \ \gamma\mathfrak{S}\cap \mathfrak{T}\neq \emptyset \}$ is finite.
        \item There exists a finite set $C\subset G$ such that 
        \[
            G = \Gamma \cdot C\cdot \mathfrak{S}. 
        \]
      \end{enumerate}
\end{lemma}
When passing from $G$ to the $G$-homogeneous variety $D$, we also have:
\begin{lemma}\label{L:SiegelProperties2}
    Let $\mathfrak{S}\cdot F_1, \mathfrak{T}\cdot F_2$ be Siegel sets in $D$, and $\Gamma$ be an arithmetic subgroup of $G$. 
    \begin{enumerate}
        \item The set $\{\gamma\in \Gamma \ | \ \gamma\mathfrak{S}\cdot F_1\cap \mathfrak{T}\cdot F_2\neq \emptyset \}$ is finite.
        \item There exists a finite set $C\subset G$ such that 
        \[
            D = \Gamma \cdot C\cdot \mathfrak{S}\cdot F_1. 
        \]
      \end{enumerate}
\end{lemma}

\begin{remark}
    Lemma \ref{L:SiegelProperties} and \ref{L:SiegelProperties2} implies Siegel sets are fundamental sets up to finite translations. 
\end{remark}

\subsection{Levi decomposition and generalized Siegel sets}

We generalize the notion of Siegel sets for certain non-reductive algebraic groups. Let $\bfQ$ be a $\bQ$-algebraic group admitting a Levi-type extension:
\[
1\rightarrow \bfU\rightarrow \bfQ\rightarrow \bfL\rightarrow 1   
\]
where $\bfU \trianglelefteq\bfQ$ is the unipotent radical of $\bfQ$ which is a unipotent $\bQ$-algebraic group, and $\bfL\simeq \bfQ/\bfU$ is a $\bQ$-reductive group. A typical example is to look at a $\bQ$-parabolic subgroup of $\bfG$ and its Levi decomposition.

Suppose $\bfL_0\leq \bfQ$ is a lift of $\bfL$ defined over $\bR$, such that there is a Levi-type decomposition of $\bR$-algebraic groups
\[
U\rtimes L_0\xrightarrow{\simeq} Q.
\]
Let $\fS_L\subset L$ be a standard Siegel set in $L$ and $\fS_{L,0}\subset L_0$ be its lift in $L_0$. Let $\Omega_U\subset U$ be an open bounded subset. We call the set
\[
\fS=\Omega_U\cdot \fS_{L,0}\subset Q
\]
a generalized Siegel set of $Q$. Let $\Gamma_Q\leq \bfQ(\bQ)$ be a subgroup commensurable with $\bfQ(\bZ)$. In this case we call $\Gamma_Q\leq \bfQ(\bQ)$ is an arithmetic subgroup. We have the following Siegel properties for generalized Siegel sets:

\begin{lemma}\label{L:GeneralizedsiegelpropertiesQ}
    Let $\mathfrak{S}, \mathfrak{T}$ be generalized Siegel sets in $Q$ corresponds to the same Levi decomposition $U\rtimes L_0\xrightarrow{\simeq} Q$, and $\Gamma_Q\leq \bfQ(\bQ)$ be an arithmetic subgroup.
    \begin{enumerate}
        \item The set $\{\gamma\in \Gamma_Q \ | \ \gamma\mathfrak{S}\cap \mathfrak{T}\neq \emptyset \}$ is finite.
        \item There exists a finite set $C\subset Q$ such that 
        \begin{equation*}
            Q = \Gamma_Q \cdot C\cdot \fS.
        \end{equation*}
    \end{enumerate}
\end{lemma}

 Let $D$ be a $Q$-homogeneous variety such that $D\simeq Q/M$ for some compact subgroup $M$. A typical example is to consider $D$ as the mixed period domain of graded polarizable $\bZ$-mixed Hodge structures. 

\begin{definition}
    Suppose $\fS=\Omega_U\cdot \fS_{L,0}\subset Q$ is a generalized Siegel set in $Q$ where $\fS_L$ is a Siegel set in $L$ corresponding to Siegel triple $(P_L, A_L, K_L)$. If $K_{L,0}:=K\cap Q$ is a lift of $K_L$ and $M\leq K_{L,0}$, we call the subset
    \[
    \fS\cdot [M]=  \Omega_U\cdot \fS_{L,0}\cdot [M]\subset D  
    \]
    as a generalized Siegel set in $D$.
\end{definition}

\begin{lemma}\label{L:GeneralizedsiegelpropertiesD}
    Let $\mathfrak{S}\cdot F_1$ and $\mathfrak{T}\cdot F_2$ be generalized Siegel sets in $D$ corresponding to the same Levi decomposition $U\rtimes L_0\xrightarrow{\simeq} Q$, and $\Gamma_Q\leq \bfQ(\bQ)$ be an arithmetic subgroup.
    \begin{enumerate}
        \item The set $\{\gamma\in \Gamma_Q \ | \ \gamma\mathfrak{S}\cdot F_1\cap \mathfrak{T}\cdot F_2\neq \emptyset \}$ is finite.
        \item There exist finite subsets $C\subset Q$ such that 
        \begin{equation*}
            D = \Gamma_Q \cdot C\cdot \fS\cdot F_1.
        \end{equation*}
    \end{enumerate}
\end{lemma}

\begin{proof}[Proof of Lemma \ref{L:GeneralizedsiegelpropertiesQ} and \ref{L:GeneralizedsiegelpropertiesD}]
If $\bfL_0$ is defined over $\bQ$, the lemmas are exactly Lemma \ref{L:GFundSet}. Now fix a $\bQ$-Levi subgroup $\bfL_\bQ\leq \bfQ$, then there exists some $u_0\in U$ such that $L_0=u_0L_\bQ u_0^{-1}$.

Let $\fS_{L,\bQ}\subset L_\bQ$ is the lift of $\fS_0$ in $L_\bQ$, then we may write
\[
\fS=\Omega_1\cdot \fS_{L,0}=\Omega_1\cdot u_0\fS_{L,\bQ}u_0^{-1}, \ 
\fT=\Omega_2\cdot \fT_{L,0}=\Omega_2\cdot u_0\fT_{L,\bQ}u_0^{-1}.
\]
According to Lemma \ref{L:GFundSet}, after passing to a finite-index subgroup, we may assume $\Gamma_Q=\Gamma_U\rtimes\Gamma_L$ with $\Gamma_U\leq \bfU(\bQ)$ is discrete and cocompact, and $\Gamma_L\leq \bfL(\bQ)$ is arithmetic. For $\gamma=\gamma_u\gamma_l$ with $\gamma_u\in \Gamma_U$ and $\gamma_l\in \Gamma_L$ we have
\begin{eqnarray*}\label{E:GeneSiegelIntersect}
    \gamma\fS \cap \fT &=&\gamma_u\gamma_l\cdot\Omega_1\cdot \fS_{L,0} \cap \Omega_2\cdot \fT_{L,0}\\
    &=& (\gamma_u\gamma_l\Omega_1u_0\gamma_l^{-1})\cdot(\gamma_l\fS_{L,\bQ})\cdot u_0^{-1}\cap (\Omega_2u_0)\cdot \fT_{L,\bQ}\cdot u_0^{-1}.\\
    &=&\{(\gamma_u\gamma_l\Omega_1u_0\gamma_l^{-1}\cdot\gamma_l\fS_{L,\bQ})\cap (\Omega_2u_0\cdot \fT_{L,\bQ})\}\cdot u_0^{-1}.
\end{eqnarray*}
By comparing each components in the Levi decomposition, we conclude that there are only finitely many $\gamma_l\in \Gamma_L$ such that $\gamma_l\fS_{L,\bQ}\cap \fT_{L,\bQ}\neq\emptyset$, and for any fixed such $\gamma_l\in \Gamma_L$, there are only many $\gamma_u\in \Gamma_U$ such that $\gamma_u\gamma_l\Omega_1u_0\gamma_l^{-1}\cap \Omega_2u_0\neq\emptyset$. This proves (1) of Lemma \ref{L:GeneralizedsiegelpropertiesQ}.

Fix an open bounded $\Omega_0\subset U$. By Lemma \ref{L:GFundSet} there exists $C\subset Q$ with 
\[
Q=\Gamma_Q\cdot C\cdot \Omega_0\cdot\fS_{L,\bQ}
\]
which is equivalent to
\[
Q=\Gamma_Q\cdot C\cdot \Omega_0\cdot\fS_{L,\bQ}\cdot u_0^{-1}=\Gamma_Q\cdot C\cdot (\Omega_0u_0^{-1})\cdot \fS_{L,0} 
\]
Therefore (2) of Lemma \ref{L:GeneralizedsiegelpropertiesQ} follows as we may replace $C$ by $C\cdot C'$ with $C'\cdot\Omega_1\supset\Omega_0u_0^{-1}$ if necessary. This completes the proof of Lemma \ref{L:GeneralizedsiegelpropertiesQ}.

The progression from Lemma \ref{L:GeneralizedsiegelpropertiesQ} to Lemma \ref{L:GeneralizedsiegelpropertiesD} is no different than the one from Lemma \ref{L:SiegelProperties} to Lemma \ref{L:SiegelProperties2} which is well-known in the classical reduction theory.
\end{proof}

\subsection{Generalized Siegel sets for pairs}

For the sake of completeness, in this section we provide an alternative version of the main theorem of \cite{MR3803710} and \cite{MR4593766} for generalized Siegel sets (Theorem \ref{T:OrrGeneralized}). This theorem is not used in the rest of the paper.

Assume $\bfQ_1\leq \bfQ_2$ is a pair of $\bQ$-algebraic groups admitting compatible Levi-type extension sequence:
\begin{equation}\label{E:cptbleLeviExt}
\begin{tikzcd}
1 \arrow[r] \arrow[d, equal] & \bfU_1 \arrow[r] \arrow[d, hookrightarrow] & \bfQ_1 \arrow[r] \arrow[d, hookrightarrow] & \bfL_1 \arrow[r] \arrow[d, hookrightarrow] & 1 \arrow[d, equal] \\1 \arrow[r] & \bfU_2 \arrow[r] & \bfQ_2 \arrow[r] & \bfL_2 \arrow[r] & 1
\end{tikzcd}
\end{equation}
Suppose $(P_i, A_i, K_i)_{i=1,2}$ are Siegel triples of $\bfL_i$, and $L_{2,0}\leq Q_2$ is a lift of $L_2$ such that $L_{1,0}:=L_{2,0}\cap Q_1$ is a lift of $L_1$. 
\begin{theorem}\label{T:OrrGeneralized}
    Suppose $(P_i, A_i, K_i)_{i=1,2}$ satisfies the following relations:
    \begin{enumerate}
        \item $K_2\cap L_1=K_1$,
        \item The Cartan involution $\theta_{2}$ of $K_2$ on $L_2$ stabilizes $A_1$.
    \end{enumerate}
    Then there exist Siegel sets $\fS_{L,i}$ of $L_i$ with respect to the Siegel triple $(P_i, A_i, K_i)$, their lifts $\fS_{L,i,0}$ in $L_{i,0}$, bounded subsets $\Omega_i\subset U_i$ and a finite set $C\subset L_{2,0}$ such that the sets
    \[
    \fS_i:= \Omega_i\cdot \fS_{L,i,0} 
    \]
    are generalized Siegel sets in $Q_i$ and satisfy
    \[
    \fS_1\subset C\cdot \fS_2.    
    \]
\end{theorem}

\begin{proof}
    We may apply Theorem \ref{T:OrrSch} to $L_{1,0}\leq L_{2,0}$ and the corresponding Siegel triples, then the lemma follows directly from the Definition of generalized Siegel sets and Lemma \ref{L:GeneralizedsiegelpropertiesQ}.
\end{proof}


\def\cprime{$'$} \def\Dbar{\leavevmode\lower.6ex\hbox to 0pt{\hskip-.23ex
  \accent"16\hss}D}


\begin{thebibliography}{CEZGT14}

\bibitem[ACT02]{MR1910264}
Daniel Allcock, James~A. Carlson, and Domingo Toledo.
\newblock The complex hyperbolic geometry of the moduli space of cubic
  surfaces.
\newblock {\em J. Algebraic Geom.}, 11(4):659--724, 2002.

\bibitem[ACT11]{MR2789835}
Daniel Allcock, James~A. Carlson, and Domingo Toledo.
\newblock The moduli space of cubic threefolds as a ball quotient.
\newblock {\em Mem. Amer. Math. Soc.}, 209(985):xii+70, 2011.

\bibitem[Ale96]{AlexeevBBsing}
Valery Alexeev.
\newblock Log canonical singularities and complete moduli of stable pairs.
\newblock arXiv:9608013, 1996.

\bibitem[AMRT75]{MR0457437}
A.~Ash, D.~Mumford, M.~Rapoport, and Y.~Tai.
\newblock {\em Smooth compactification of locally symmetric varieties}.
\newblock Math. Sci. Press, Brookline, Mass., 1975.
\newblock Lie Groups: History, Frontiers and Applications, Vol. IV.

\bibitem[BB66]{MR0216035}
W.~L. Baily, Jr. and A.~Borel.
\newblock Compactification of arithmetic quotients of bounded symmetric
  domains.
\newblock {\em Ann. of Math. (2)}, 84:442--528, 1966.

\bibitem[BBT23]{MR4557401}
Benjamin Bakker, Yohan Brunebarbe, and Jacob Tsimerman.
\newblock o-minimal {GAGA} and a conjecture of {G}riffiths.
\newblock {\em Invent. Math.}, 232(1):163--228, 2023.

\bibitem[BGS+23]{MR4614603}
B.~Bakker, T.~W. Grimm, C.~Schnell, and J.~Tsimerman.
\newblock Finiteness for self-dual classes in integral variations of {Hodge} structure.
\newblock {\em Épijournal de Géométrie Algébrique}, Special volume in honour of Claire Voisin, Art. 1, 2023.
\newblock DOI: 10.46298/epiga.2023.specialvolumeinhonourofclairevoisin.9626.

\bibitem[BKT20]{MR4155216}
Benjamin Bakker, Bruno Klingler, and Jacob Tsimerman.
\newblock Tame topology of arithmetic quotients and algebraicity of {H}odge loci.
\newblock {\em Journal of the American Mathematical Society}, 33(4):917--939, 2020.
\newblock \href{https://doi.org/10.1090/jams/952}{doi:10.1090/jams/952}.

\bibitem[BHC62]{MR0147566}
Armand Borel and Harish-Chandra.
\newblock Arithmetic subgroups of algebraic groups.
\newblock {\em Ann. of Math. (2)}, 75:485--535, 1962.

\bibitem[BJ06]{MR2189882}
Armand Borel and Lizhen Ji.
\newblock {\em Compactifications of symmetric and locally symmetric spaces}.
\newblock Mathematics: Theory \& Applications. Birkh{\"a}user Boston, Inc., Boston, MA, 2006.

\bibitem[BFNP09]{MR2534102}
Patrick Brosnan, Hao Fang, Zhaohu Nie, and Gregory Pearlstein.
\newblock Singularities of admissible normal functions.
\newblock {\em Invent. Math.}, 177(3):599--629, 2009.
\newblock With an appendix by Najmuddin Fakhruddin.

\bibitem[Bor66]{MR0204533}
Armand Borel.
\newblock Reduction theory for arithmetic groups.
\newblock In {\em Algebraic {G}roups and {D}iscontinuous {S}ubgroups ({P}roc.
  {S}ympos. {P}ure {M}ath., {B}oulder, {C}olo., 1965)}, pages 20--25. Amer.
  Math. Soc., Providence, RI, 1966.

\bibitem[Bor69]{MR0244260}
Armand Borel.
\newblock {\em Introduction aux groupes arithm{\'e}tiques}, Publications de l'Institut de Math{\'e}matique de l'Universit{\'e} de Strasbourg, XV. Actualit{\'e}s Scientifiques et Industrielles, No. 1341.
\newblock Hermann, Paris, 1969.

\bibitem[Bor72]{MR0338456}
Armand Borel.
\newblock Some metric properties of arithmetic quotients of symmetric spaces
  and an extension theorem.
\newblock {\em J. Differential Geometry}, 6:543--560, 1972.
\newblock Collection of articles dedicated to S. S. Chern and D. C. Spencer on their sixtieth birthdays.

\bibitem[Bor91]{MR1102012}
Armand Borel.
\newblock {\em Linear algebraic groups}, volume 126 of {\em Graduate Texts in
  Mathematics}.
\newblock Springer-Verlag, New York, second edition, 1991.

\bibitem[Bor97]{MR1416355}
Ciprian Borcea.
\newblock {$K3$} surfaces with involution and mirror pairs of {C}alabi-{Y}au
  manifolds.
\newblock In {\em Mirror symmetry, {II}}, volume~1 of {\em AMS/IP Stud. Adv.
  Math.}, pages 717--743. Amer. Math. Soc., Providence, RI, 1997.

\bibitem[BPR17]{MR3751291}
P.~Brosnan, G.~Pearlstein, and C.~Robles.
\newblock Nilpotent cones and their representation theory.
\newblock In Lizhen Ji, editor, {\em Hodge theory and {$L^2$}-analysis},
  volume~39 of {\em Adv. Lect. Math. (ALM)}, pages 151--205. Int. Press,
  Somerville, MA, 2017.
\newblock arXiv:1602.00249.

\bibitem[Cat14a]{MR3290123}
Eduardo Cattani.
\newblock Introduction to {K}\"{a}hler manifolds.
\newblock In {\em Hodge theory}, volume~49 of {\em Math. Notes}, pages 1--69.
  Princeton Univ. Press, Princeton, NJ, 2014.

\bibitem[Cat14b]{MR3290129}
Eduardo Cattani.
\newblock Introduction to variations of {H}odge structure.
\newblock In {\em Hodge theory}, volume~49 of {\em Math. Notes}, pages
  297--332. Princeton Univ. Press, Princeton, NJ, 2014.

\bibitem[CCK80]{MR605337}
James~A. Carlson, Eduardo~H. Cattani, and Aroldo~G. Kaplan.
\newblock Mixed {H}odge structures and compactifications of {S}iegel's space
  (preliminary report).
\newblock In {\em Journ\'{e}es de {G}\'{e}ometrie {A}lg\'{e}brique d'{A}ngers,
  {J}uillet 1979/{A}lgebraic {G}eometry, {A}ngers, 1979}. 77--105, 1980.


\bibitem[CEZGT14]{MR3288678}
Eduardo Cattani, Fouad El~Zein, Phillip~A. Griffiths, and L{\^e}~D{\~u}ng
  Tr{{\'a}}ng, editors.
\newblock {\em Hodge theory}, volume~49 of {\em Mathematical Notes}.
\newblock Princeton University Press, Princeton, NJ, 2014.

\bibitem[CFPR22]{CFPR2022}
Stephen Coughlan, Marco Franciosi, Rita Pardini, and S\"{o}nke Rollenske.
\newblock Degeneration of {H}odge structures on {I}-surfaces.
\newblock arXiv:2209.07150, 2022.

\bibitem[Che23]{Chen1221}
Chongyao Chen.
\newblock Completion of a period map of hodge type (1,2,2,1).
\newblock arXiv:2311.10212, 2023.

\bibitem[CK81]{MR590823}
Eduardo Cattani and Aroldo Kaplan.
\newblock On the local monodromy of a variation of {H}odge structure.
\newblock {\em Bull. Amer. Math. Soc. (N.S.)}, 4(1):116--118, 1981.

\bibitem[CK82]{MR664326}
Eduardo Cattani and Aroldo Kaplan.
\newblock Polarized mixed {H}odge structures and the local monodromy of a
  variation of {H}odge structure.
\newblock {\em Invent. Math.}, 67(1):101--115, 1982.

\bibitem[CKS86]{MR840721}
Eduardo Cattani, Aroldo Kaplan, and Wilfried Schmid.
\newblock Degeneration of {H}odge structures.
\newblock {\em Ann. of Math. (2)}, 123(3):457--535, 1986.

\bibitem[Den22]{MR4441155}
Haohua Deng.
\newblock Extension of period maps by polyhedral fans.
\newblock {\em Adv. Math.}, 406:Paper No. 108532, 36, 2022.

\bibitem[Den23]{deng2023}
Haohua Deng.
\newblock Space of nilpotent orbits and extension of period maps (i): the
  weight 3 {C}alabi--{Y}au types.
\newblock arXiv:2306.11254, 2023.

\bibitem[GG06]{MR2313336}
Mark Green and Phillip Griffiths.
\newblock Algebraic cycles and singularities of normal functions. {II}.
\newblock In {\em Inspired by {S}. {S}. {C}hern}, volume~11 of {\em Nankai
  Tracts Math.}, pages 179--268. World Sci. Publ., Hackensack, NJ, 2006.

\bibitem[GG07]{MR2385303}
Mark Green and Phillip Griffiths.
\newblock Algebraic cycles and singularities of normal functions.
\newblock In {\em Algebraic cycles and motives. {V}ol. 1}, volume 343 of {\em
  London Math. Soc. Lecture Note Ser.}, pages 206--263. Cambridge Univ. Press,
  Cambridge, 2007.

\bibitem[GGK10]{MR2601630}
Mark Green, Phillip Griffiths, and Matt Kerr.
\newblock N\'{e}ron models and limits of {A}bel-{J}acobi mappings.
\newblock {\em Compos. Math.}, 146(2):288--366, 2010.

\bibitem[GGK12]{MR2918237}
Mark Green, Phillip Griffiths, and Matt Kerr.
\newblock {\em Mumford-{T}ate groups and domains: their geometry and
  arithmetic}, volume 183 of {\em Annals of Mathematics Studies}.
\newblock Princeton University Press, Princeton, NJ, 2012.

\bibitem[GGLR20]{GGLR}
M.~Green, P.~Griffiths, R.~Laza, and C.~Robles.
\newblock Period mappings and properties of the augmented {H}odge line bundle.
\newblock arXiv:1708.09523, 2020.

\bibitem[GGR21]{GGR2LB}
Mark Green, Phillip Griffiths, and Colleen Robles.
\newblock Geometric properties of line bundles in hodge theory.
\newblock arXiv:2102.06310, 2021.

\bibitem[GGR22]{GGRinfty}
Mark Green, Phillip Griffiths, and Colleen Robles.
\newblock The global asymptotic structure of period mappings.
\newblock arXiv:2010.06720, 2022.

\bibitem[Gri70a]{MR0282990}
Phillip~A. Griffiths.
\newblock Periods of integrals on algebraic manifolds. {III}. {S}ome global
  differential-geometric properties of the period mapping.
\newblock {\em Inst. Hautes \'Etudes Sci. Publ. Math.}, pages 125--180, 1970.

\bibitem[Gri70b]{MR0258824}
Phillip~A. Griffiths.
\newblock Periods of integrals on algebraic manifolds: {S}ummary of main
  results and discussion of open problems.
\newblock {\em Bull. Amer. Math. Soc.}, 76:228--296, 1970.

\bibitem[Hay11]{MR2832807}
Tatsuki Hayama.
\newblock N\'{e}ron models of {G}reen-{G}riffiths-{K}err and log {N}\'{e}ron
  models.
\newblock {\em Publ. Res. Inst. Math. Sci.}, 47(3):803--824, 2011.







\bibitem[Kas86]{MR866665}
Masaki Kashiwara.
\newblock A study of variation of mixed {H}odge structure.
\newblock {\em Publ. Res. Inst. Math. Sci.}, 22(5):991--1024, 1986.

\bibitem[Kli90]{MR1046630}
Helmut Klingen.
\newblock {\em Introductory Lectures on Siegel Modular Forms}, volume~20 of {\em Cambridge Studies in Advanced Mathematics}.
\newblock Cambridge University Press, Cambridge, 1990.

\bibitem[KNU10]{MR2721860}
Kazuya Kato, Chikara Nakayama, and Sampei Usui.
\newblock N\'{e}ron models in log mixed {H}odge theory by weak fans.
\newblock {\em Proc. Japan Acad. Ser. A Math. Sci.}, 86(8):143--148, 2010.

\bibitem[KNU13]{MR3084721}
Kazuya Kato, Chikara Nakayama, and Sampei Usui.
\newblock Classifying spaces of degenerating polarized {H}odge structures, iii:
  spaces of nilpotent orbits.
\newblock {\em J. Algebraic Geom.}, 22(4):671--772, 2013.

\bibitem[Kol13]{MR3184176}
J{{\'a}}nos Koll{{\'a}}r.
\newblock Moduli of varieties of general type.
\newblock In {\em Handbook of moduli. {V}ol. {II}}, volume~25 of {\em Adv.
  Lect. Math. (ALM)}, pages 131--157. Int. Press, Somerville, MA, 2013.

\bibitem[KP11]{MR2796415}
Matt Kerr and Gregory Pearlstein.
\newblock An exponential history of functions with logarithmic growth.
\newblock In {\em Topology of stratified spaces}, volume~58 of {\em Math. Sci.
  Res. Inst. Publ.}, pages 281--374. Cambridge Univ. Press, Cambridge, 2011.

\bibitem[KP16]{MR3474815}
Matt Kerr and Gregory Pearlstein.
\newblock Boundary components of {M}umford-{T}ate domains.
\newblock {\em Duke Math. J.}, 165(4):661--721, 2016.

\bibitem[KPR19]{MR4012553}
M.~Kerr, G.~J. Pearlstein, and C.~Robles.
\newblock Polarized relations on horizontal {$\rm SL(2)$}'s.
\newblock {\em Doc. Math.}, 24:1295--1360, 2019.


\bibitem[KU09]{MR2465224}
Kazuya Kato and Sampei Usui.
\newblock {\em Classifying spaces of degenerating polarized {H}odge
  structures}, volume 169 of {\em Annals of Mathematics Studies}.
\newblock Princeton University Press, Princeton, NJ, 2009.

\bibitem[Kos59]{MR0114875}
Bertram Kostant.
\newblock{The principal three-dimensional subgroup and the {B}etti numbers of a complex simple {L}ie group.}
\newblock {\em Amer. J. Math.}, 81:973--1032, 1959.

\bibitem[Laz16]{MR3495110}
Radu Laza.
\newblock Perspectives on the construction and compactification of moduli
  spaces.
\newblock In {\em Compactifying moduli spaces}, Adv. Courses Math. CRM
  Barcelona, pages 1--39. Birkh\"{a}user/Springer, Basel, 2016.

\bibitem[LPZ18]{MR3886178}
Radu Laza, Gregory Pearlstein, and Zheng Zhang.
\newblock On the moduli space of pairs consisting of a cubic threefold and a
  hyperplane.
\newblock {\em Adv. Math.}, 340:684--722, 2018.

\bibitem[Mil20]{MR2914941}
J.~S. Milne.
\newblock Classification of the {Mumford--Tate} groups of rational polarizable {Hodge} structures.
\newblock Available at \url{https://www.jmilne.org/math/articles/CMT.pdf}, 2020.

\bibitem[Mos56]{MR0092928}
G.~D. Mostow.
\newblock Fully reducible subgroups of algebraic groups.
\newblock {\em Amer. J. Math.}, 78:200--221, 1956.

\bibitem[Mum75]{MR0485875}
David Mumford.
\newblock A new approach to compactifying locally symmetric varieties.
\newblock In {\em Discrete subgroups of {L}ie groups and applications to moduli
  ({I}nternat. {C}olloq., {B}ombay, 1973)}, pages 211--224. 1975.

\bibitem[Nak12]{MR3086749}
Chikara Nakayama.
\newblock Log {N}\'{e}ron models over surfaces.
\newblock {\em J. Math. Sci. Univ. Tokyo}, 19(4):613--659, 2012.

\bibitem[Nak15]{MR3532114}
Chikara Nakayama.
\newblock Log {N}\'{e}ron models over surfaces, {II}.
\newblock {\em Hokkaido Math. J.}, 44(3):365--385, 2015.

\bibitem[Orr18]{MR3803710}
Martin Orr.
\newblock Height bounds and the {S}iegel property.
\newblock {\em Algebra \& Number Theory}, 12(2):455--478, 2018.
\newblock \href{https://doi.org/10.2140/ant.2018.12.455}.

\bibitem[OS23]{MR4593766}
Martin Orr and Christian Schnell.
\newblock Correction to the article {H}eight bounds and the {S}iegel property.
\newblock {\em Algebra Number Theory}, 17(6):1231--1237, 2023.

\bibitem[PS08]{MR2393625}
Chris A.~M. Peters and Joseph H.~M. Steenbrink.
\newblock {\em Mixed {H}odge structures}, volume~52 of {\em Ergebnisse der
  Mathematik und ihrer Grenzgebiete. 3. Folge. A Series of Modern Surveys in
  Mathematics [Results in Mathematics and Related Areas. 3rd Series. A Series
  of Modern Surveys in Mathematics]}.
\newblock Springer-Verlag, Berlin, 2008.

\bibitem[Rob16]{MR3505643}
Colleen Robles.
\newblock Classification of horizontal {${\rm SL}(2)$}s.
\newblock {\em Compos. Math.}, 152(5):918--954, 2016.

\bibitem[Roh09]{MR2510071}
Jan~Christian Rohde.
\newblock {\em Cyclic coverings, {C}alabi-{Y}au manifolds and complex
  multiplication}, volume 1975 of {\em Lecture Notes in Mathematics}.
\newblock Springer-Verlag, Berlin, 2009.

\bibitem[Sai96]{MR1374710}
Morihiko Saito.
\newblock Admissible normal functions.
\newblock {\em J. Algebraic Geom.}, 5(2):235--276, 1996.

\bibitem[Sat60]{MR0170356}
Ichir{\c{o}} Satake.
\newblock On compactifications of the quotient spaces for arithmetically
  defined discontinuous groups.
\newblock {\em Ann. of Math. (2)}, 72:555--580, 1960.

\bibitem[Sch73]{MR0382272}
Wilfried Schmid.
\newblock Variation of {H}odge structure: the singularities of the period
  mapping.
\newblock {\em Invent. Math.}, 22:211--319, 1973.

\bibitem[Sch12]{MR2897692}
Christian Schnell.
\newblock Complex analytic {N}\'{e}ron models for arbitrary families of
  intermediate {J}acobians.
\newblock {\em Invent. Math.}, 188(1):1--81, 2012.


\bibitem[Som73]{MR0324078}
Andrew~J. Sommese.
\newblock Some algebraic properties of the image of the period mapping.
\newblock {\em Rice Univ. Studies}, 59(2):123--128, 1973.
\newblock Complex analysis, 1972 (Proc. Conf., Rice Univ., Houston, Tex.,
  1972),Vol. II: Analysis on singularities.

\bibitem[Urb24]{Urb24}
David Urbanik.
\newblock Degrees of Hodge loci.
\newblock Preprint, arXiv:2412.08924, 2024.
\newblock \href{https://arxiv.org/abs/2412.08924}{arXiv:2412.08924}.

\bibitem[Usu84]{MR0771384}
Sampei Usui.
\newblock Variation of mixed {H}odge structure arising from family of
  logarithmic deformations. {II}. {C}lassifying space.
\newblock {\em Duke Math. J.}, 51(4):851--875, 1984.

\bibitem[Usu06]{MR2237263}
Sampei Usui.
\newblock Images of extended period maps.
\newblock {\em J. Algebraic Geom.}, 15(4):603--621, 2006.


\end{thebibliography}
\end{document}